\documentclass[leqno,12pt]{amsart}
\usepackage{amssymb}
\usepackage{amsmath}
\usepackage{enumerate}
\usepackage{amsfonts}
\usepackage{hyperref}
\usepackage{ulem}

\usepackage[headheight=18pt, top=25mm, bottom=25mm, left=25mm, right=25mm]{geometry}

\usepackage{tikz}
\usepackage{pgfplots}

\definecolor{darkgreen}{RGB}{45, 119, 75}

\newtheorem{theorem}{Theorem}[section]
\newtheorem{corollary}[theorem]{Corollary}
\newtheorem{lemma}[theorem]{Lemma}
\newtheorem{proposition}[theorem]{Proposition}
\newtheorem{remark}[theorem]{Remark}
\newtheorem{definition}[theorem]{Definition}

\newtheorem{example}{Example}[section]

\numberwithin{equation}{section}

\hypersetup{
    colorlinks=true,
    linkcolor= blue,
    citecolor =cyan,
    urlcolor = teal,
}

\begin{document}

\title[Lipschitz spaces]{On Lipschitz spaces  in the Dunkl setting - \\
semigroup approach }


\author[Jacek Dziubański]{Jacek Dziubański}
\author[Agnieszka Hejna]{Agnieszka Hejna}
\begin{abstract} 
Let $\{P_t\}_{t>0}$ be the Dunkl-Poisson semigroup associated with a root system $R\subset \mathbb R^N$ and a multiplicity function $k\geq 0$. Analogously to the classical theory, we say that a bounded measurable function $f$ defined on $\mathbb R^N$ belongs to the inhomogeneous Lipschitz space $\Lambda_k^\beta$, $\beta>0$, if 
$$\sup_{t>0} t^{m-\beta} \Big\|\frac{d^m}{dt^m} P_tf\Big\|_{L^\infty}<\infty,$$ 
where $m=[\beta]+1$. We prove that the spaces $\Lambda^\beta_k$ coincide with the classical Lipschitz spaces. 

\end{abstract}

\address{Jacek Dziubański, Uniwersytet Wroc\l awski,
Instytut Matematyczny,
Pl. Grunwaldzki 2,
50-384 Wroc\l aw,
Poland}
\email{jdziuban@math.uni.wroc.pl}

\address{Agnieszka Hejna, Uniwersytet Wroc\l awski,
Instytut Matematyczny,
Pl. Grunwaldzki 2,
50-384 Wroc\l aw,
Poland 
\&
Department of Mathematics,
Rutgers University,
Piscataway, NJ 08854-8019, USA}
\email{hejna@math.uni.wroc.pl}

\maketitle

\section{Introduction}

{    For $0<\beta<1$, the the classical inhomogeneous Lipschitz space $\Lambda^\beta(\mathbb R^N)$ on the Euclidean space $\mathbb R^N$ is defined as 
 \begin{equation}\label{eq:Lip_class}
     \Lambda^\beta(\mathbb R^N)=\Big\{ f: \;f: \mathbb R^N\to \mathbb C, \ \|f\|_{L^\infty} +\sup_{\mathbf{x}\ne \mathbf{x}'}\frac{|f(\mathbf x)-f(\mathbf x')|}{\|\mathbf x-\mathbf x'\|^\beta}=:\| f\|_{\Lambda^\beta (\mathbb R^N)}<\infty\Big\}.
 \end{equation}
 In Taibleson \cite{Taibleson1}, the author used the Poisson integral (semigroup) 
 \begin{equation}
     \label{eq:Poisson_int}
     f(\mathbf x,t)=c_N^{-1}\int_{\mathbb{R}^N} f(\mathbf{y})\frac{t}{(t^2+\|\mathbf x-\mathbf{y}\|^2)^{(N+1)/2}}\, d\mathbf{y} 
 \end{equation} 
 to study properties of the Lipschitz spaces on $\Lambda^\beta(\mathbb R^N)$ (see also \cite{Stein_singular}). To be more precise, for  any fixed positive integer $m$,  the norm $\| f\|_{\Lambda^\beta (\mathbb R^N)}$ is equivalent to 
 \begin{equation}
     \label{eq:equiv_cond_1}
\|f\|_{L^\infty}+     \sup_{t>0} t^{m-\beta}\Big\|\frac{d^m}{dt^m} f(\mathbf{x},t)\Big\|_{L^\infty}, 
 \end{equation}
see \cite[Theorems 3 and 4]{Taibleson1} and \cite[Chapter V, Proposition 7 and Lemma 5]{Stein_singular} . 

The theorem gives rise to extend the notion of the Lipschitz spaces for all positive parameters $\beta>0$.  So,  for  $\beta>0$, let $m$ be the smallest integer bigger than $\beta$. We say that $f\in\Lambda^{\beta}(\mathbb R^N)$, if \eqref{eq:equiv_cond_1} is finite (then \eqref{eq:equiv_cond_1} is taken as the norm in the space). 
Equivalently one can consider the heat semigroup $\{e^{t\Delta}\}_{t \geq 0}$ and take 
\begin{equation}
    {\| f\|_{L^\infty} +  \sup_{t>0} t^{m-\beta/2}\Big\|\frac{d^m}{d t^m } e^{t\Delta} f\Big\|_{L^\infty}} 
\end{equation}
as the norm in the Lipschitz space $\Lambda^\beta(\mathbb R^N)$ (see \cite[Theorem 7]{Taibleson1}). 

It turns out (see \cite[Theorem 4]{Taibleson1}) that if $0<\beta<2$, then the norm \eqref{eq:equiv_cond_1} (with any fixed $m\geq 2$) is equivalent to 
\begin{equation}
    \label{eq:Zygmund_cond} 
    \|f\|_{L^\infty} +
  \sup_{\mathbf x\in\mathbb R^N} \sup_{0\ne  \mathbf y \in \mathbb R^N } \frac{\| f(\mathbf x+\mathbf y)+f(\mathbf x-\mathbf y)- 2f(\mathbf x)\|}{\| \mathbf y\|^\beta} .     
\end{equation}
{ Such  spaces can be thought as inhomogeneous  generalized   Zygmund classes. }

Further,  if $\beta>1$, then $f\in \Lambda^\beta(\mathbb R^N)$ if and only if $f\in L^\infty$ and $\partial_j f\in \Lambda^{\beta-1}(\mathbb R^N)$ for $j=1,2,...,N$ (see e.g. \cite[Chapter V, Proposition 9]{Stein_singular}).   Moreover, 
\begin{equation}
    \label{eq:step-less}
\|f\|_{\Lambda^\beta(\mathbb R^N)} \sim \| f\|_{L^\infty} +\sum_{j=1}^N \| \partial_j f\|_{\Lambda^{\beta-1}(\mathbb R^N)}. 
\end{equation}

Additionally, the Bessel potential 
$(I-\Delta)^{\gamma/2}$  is an isomorphism of the space $\Lambda^{\beta}(\mathbb R^N)$ onto $\Lambda^{\beta+\gamma}(\mathbb R^N)$ (see \cite[Theorem 5]{Taibleson1}). 

The aim of this paper is to study the inhomogeneous Lipschitz spaces in the Dunkl setting. On the Euclidean space $\mathbb R^N$ equipped with  a root system $R$ and a 
 multiplicity function $k\geq 0$, we consider the Dunkl Laplace operator 
$$ \Delta_k=\sum_{j=1}^N D_j^2,$$
where 
$$D_j f(\mathbf{x})={\partial_{j }f(\mathbf{x})} +\sum_{\alpha\in R} \frac{k(\alpha)}{2} \langle \alpha,e_j\rangle \frac{f(\mathbf{x})-f(\sigma_\alpha (\mathbf{x}))}{\langle \alpha,\mathbf{x}\rangle} $$ are the Dunkl  operators.  The operator $\Delta_k$ generates a contraction  semigroup $H_t=e^{t\Delta_k}$ on $L^p(dw)$, $1\leq p\leq \infty$ (strongly continuous for $1\leq p<\infty)$,   where 
 $$ dw(\mathbf{x})= \prod_{\alpha\in R} |\langle \alpha,\mathbf{x}\rangle|^{k(\alpha)}\, d\mathbf{x}.$$ 
 The semigroup has  a unique   extension to a uniformly bounded holomorphic  semigroup on any sector $\boldsymbol{\Delta}_\delta=\{ z\in \mathbb C: |\text{arg}\, z|<\delta\}$, $0<\delta<\pi/2$. Let 

\begin{equation}\label{eq:H_sub}
    {{P}}_t:={\Gamma(1/2)^{-1}}\int_0^{\infty}e^{-u}{H}_{t^2/(4u)}\frac{du}{\sqrt{u}} 
\end{equation}
be the subordinate (Dunkl-Poisson) semigroup. We say that a measurable function $f$ defined on $\mathbb R^N$ belongs to the inhomogeneous Lipschitz space $\Lambda_k^\beta$, if 
\begin{equation}\label{eq:norm_L_Dunkl}
    \|  f\|_{{\Lambda}^{\beta}_k} :=\| f\|_{L^\infty} + \sup_{t>0} t^{m-\beta} \Big\| \frac{d^m}{dt^m} P_tf\Big\|_{L^\infty} <\infty, 
\end{equation}
where $m$ is the smallest positive integer bigger than $\beta$. 
We are in a position to state our main result, which is obtained at the very end of the paper. 
\begin{theorem}\label{teo:main_main} For any $\beta>0$, the Lipschitz space $\Lambda_k^\beta$ coincides with the classical Lipschitz space $\Lambda^\beta(\mathbb R^N)$ and the corresponding norms~\eqref{eq:norm_L_Dunkl} and \eqref{eq:equiv_cond_1}  are equivalent. 
\end{theorem}

The proof of Theorem \ref{teo:main_main} goes by   a systematic study of properties of the $\Lambda_k^\beta$ spaces which are analogue to properties  of the classical ones. Firstly, in Theorem \ref{teo:Lip_Dunk_easy}, we state that the Bessel-type potential $(I-\Delta_k)^{\gamma/2 \, *}$ is an isomorphisms of $\Lambda_k^\beta$ onto $\Lambda_k^{\beta+\gamma}$, and, for $\beta>1$, $f\in\Lambda_k^\beta$ if and only if  $f\in L^\infty$ and $D_j^*f\in \Lambda_k^{\beta-1}$, $j=1,2,...,N$.  Next, using differential properties of the Dunkl heat and the Dunkl Poisson integral kernels, we prove that for $0<\beta<1$, the spaces $\Lambda_k^\beta$ and the classical Lipschitz space $\Lambda^\beta(\mathbb R^N)$ coincide (see Theorem \ref{teo:Lambda}).   The results stated above, combined with the properties of the Bessel potential kernels, allow us to verify that  $\Lambda^\beta_k=\Lambda^\beta(\mathbb R^N)$  for $\beta>0$, $\beta\notin \mathbb Z$ (see Theorems \ref{teo:first_inclusion} and \ref{teo:inclusion_2}). Finally, we use an interpolation argument to prove the equality of the spaces for any positive integer $\beta$.

    We want to note that some theorems  can be obtained by the use of an abstract semigroup theory.  Such an approach to interpolation spaces associated with powers of generators of strongly continuous semigroups can be found e.g. in \cite{Butzer},  \cite{Triebel}, and references therein. For the reader, who is not familiar with the theory and for completeness of the paper, we present all the abstract results that are used  in proving the theorems. 
    
   We remark that our  approach to the interpolation differs a little  from that described in \cite{Butzer, Triebel}, because we interpolate between the spaces which are defined by means of actions on semigroups which are not strongly continuous on the Banach spaces under consideration. 
    So, motivated by the approach to the classical Lipschitz spaces, which is based on the action of  either the Poisson or the heat semigroup on $L^\infty$-functions,  in  Part~\ref{part1} we consider an analytic strongly continuous semigroup $\mathcal{T}_t=e^{t\mathcal{A}}$ on a Banach space $X$,  which is uniformly bounded  in a sector around the positive axis and the dual semigroup $\{\mathcal{T}_t^*\}_{t \geq 0}$ acting on  the dual space $X^*$, which is not necessarily strongly continuous in general. For such a semigroup   we define $\Lambda^\beta_{\mathcal A}$ spaces as follows. 
    \begin{definition}\label{def:Lambda}
    For $\beta>0$, let $m$ be the smallest positive integer such that $m>\beta$ . We say that $x^*\in X^*$ belongs to $\Lambda^{\beta}_{\mathcal{A}}$, if 
\begin{equation}\label{eq:alpha_decay}\Big\| \frac{d^m}{dt^m} \mathcal{T}_t^* x^*\Big\|_{X^*}\leq C t^{\beta-m}.
\end{equation} 
We equip the space $\Lambda_{\mathcal{A}}^\beta$ with the norm 
\begin{equation}
    \| x^*\|_{\Lambda_{\mathcal{A}}^\beta}:=\| x^*\|_{X^*}+\sup_{t>0} t^{m-\beta}  \Big\| \frac{d^m}{dt^m} \mathcal{T}_t^* x^*\Big\|_{X^*}.
\end{equation} 
\end{definition}
It is easy to prove that  
 $(\Lambda_{\mathcal{A}}^\beta,  \| \cdot\|_{\Lambda_{\mathcal{A}}^\beta})$ is a Banach space. 
Further, 
we define the subordinate semigroup 
$\{{\mathcal{P}}_t\}_{t \geq 0}$ according to the formula
\begin{equation}\label{eq:P_sub}
    {\mathcal{P}}_t:={\Gamma(1/2)^{-1}}\int_0^{\infty}e^{-u}\mathcal{T}_{t^2/(4u)}\frac{du}{\sqrt{u}} 
\end{equation}
acting on the Banach space $X$.
We denote by $-\sqrt{-{\mathcal{A}}}$ the infinitesimal generator of $\{\mathcal{P}_t\}_{t \geq 0}$. The semigroup $\{\mathcal{P}_t\}_{t \geq 0}$ is strongly continuous and has a unique extension to a uniformly bounded holomorphic semigroup in  the same sector. Again, we consider the dual semigroup $\{\mathcal{P}_t^*\}_{t \geq 0}$ and define  the related $\Lambda_{-\sqrt{-{\mathcal{A}}}}^\beta$ spaces, namely for $\beta>0$ and $m$ as above, we say that $x^*\in X^*$ belongs to $\Lambda_{-\sqrt{-{\mathcal{A}}}}^\beta$ if 
\begin{equation}\label{eq:alpha_root} \Big\|\frac{d^m}{dt^m} {\mathcal{P}}_t^*x^*\Big\|_{X^*} \leq Ct^{\beta-m}.
\end{equation}
Then, as in the previous case,  we set 
$$ \|x^*\|_{\Lambda_{-\sqrt{-{\mathcal{A}}}}^\beta} =\|x^*\|_{X^*} +\sup_{t>0} t^{m-\beta}  \Big\| \frac{d^m}{dt^m} \mathcal{P}_t^* x^*\Big\|_{X^*}.$$

We are now in a position to state the first result which is obtained by means of holomorphic functional calculi. 

\begin{theorem}\label{theo:Lamba=Lambda}
    For $\beta>0$, the spaces $\Lambda_{{\mathcal{A}}}^\beta$ and $\Lambda_{-\sqrt{-{\mathcal{A}}}}^{2\beta}$  coincide. Moreover, there is $C>1$ such that for all $x^* \in X^*$, we have
    \begin{equation}\label{eq:Lambda=Lambda}
        C^{-1} \| x^*\|_{\Lambda^\beta_{\mathcal A}} \leq \| x^*\|_{\Lambda^{2\beta}_{-\sqrt{-\mathcal A}}}\leq C  \| x^*\|_{\Lambda^\beta_{\mathcal A}}.
    \end{equation} 
\end{theorem} 
In continuation, for $\gamma>0$,  we consider the  Bessel type potentials 
$$ ((I-{\mathcal{A}})^{-\gamma})^*, \quad ( (I+\sqrt{-{\mathcal{A}}})^{-\gamma})^*$$
(see \eqref{eq:Bessel} and \eqref{eq:Bessel*}). 
\begin{theorem}\label{theo:Bessel}
     {Let $\beta, \gamma>0$. }The operator $((I-{\mathcal{A}})^{-\gamma})^*$ is an isomorphism of $\Lambda_{\mathcal{A}}^\beta$ onto $\Lambda_{{\mathcal{A}}}^{\gamma+\beta}$.
\end{theorem}

The theorem applied to the semigroup  $\{{\mathcal{P}}_t\}_{t>0}$ gives the following corollary. 

\begin{corollary}
    {Let $\beta, \gamma>0$. }The operator  $ ((I+\sqrt{-{\mathcal{A}}})^{-\gamma})^*$ is an isomorphism of  $\Lambda_{-\sqrt{-{\mathcal{A}}}}^\beta$ onto $\Lambda_{-\sqrt{-{\mathcal{A}}}}^{\beta+\gamma}$.
\end{corollary}

    In order to prove Theorem \ref{theo:Bessel}, we first establish the following relation of the spaces $\Lambda_{{\mathcal{A}}}^{\beta+1}$ with $\Lambda_{{\mathcal{A}}}^\beta$ and the action of ${\mathcal{A}}^*$.  

Fix $x^*\in X^*$. We say that $y^*={\mathcal{A}}^*x^*\in X^*$ in the  mild sense, if for all  $x\in X$ and all $t>0$,  
$$  \langle y^*,\mathcal{T}_tx\rangle=\langle x^*, {\mathcal{A}}\mathcal{T}_tx\rangle.$$
    
\begin{theorem}\label{theo:A_Lambda}
    Assume that $\beta>0$, $x^*\in X^*$. Then $x^*\in \Lambda_{\mathcal{A}}^{\beta+1}$ if and only if ${\mathcal{A}}^* x^*\in \Lambda_{\mathcal{A}}^\beta$, 
    where the action of ${\mathcal{A}}^*$ on $x^*$ is understood in the mild sense.  Moreover, the norms 
    $$ \| x^*\|_{\Lambda_{\mathcal{A}}^{\beta+1}} \quad \text{and} \quad \| x^*\|_{X^*} +\| {\mathcal{A}}^*x^*\|_{\Lambda_{\mathcal{A}}^\beta}$$
    are equivalent. 
\end{theorem}

We finish Part~\ref{part1} of the paper with the following theorem, which will play a crucial role in the  completing the proof of Theorem \ref{teo:main_main}. For the Banach spaces $\Lambda_{\mathcal A}^{\beta_0}$ and $\Lambda_{\mathcal A}^{\beta_1}$ and $0<\theta<1$, let $(\Lambda_{\mathcal A}^{\beta_0},\Lambda_{\mathcal A}^{\beta_1})_\theta$ and $\|x^*\|_\theta$ denote the intermediate interpolation space and the interpolation norm obtained by the $K$-method of Peetre (see Section \ref{sec:interpolation} for details).
\begin{theorem}[{cf. \cite[Section 2.7, for the classical Lipschitz spaces]{Triebel}}]\label{teo:interpolation}
    For $0<\beta_0<\beta_1$ and $0<\theta<1$, let $\beta=(1-\theta)\beta_0+\theta \beta_1$. Then 
    \begin{equation}
        (\Lambda_{\mathcal A}^{\beta_0}, \Lambda_{\mathcal A}^{\beta_1})_\theta=\Lambda_{\mathcal A}^\beta  
    \end{equation}
    and there is a constant $C>1$ such that for all $x^* \in X^{*}$ we have
    \begin{equation}
        C^{-1}\| x^*\|_{\Lambda_{\mathcal A}^\beta} \leq \| x^*\|_\theta \leq C \| x^*\|_{\Lambda_{\mathcal A}^\beta}.
    \end{equation}
\end{theorem}
 
In Part~\ref{part2}, we apply the results of Part~\ref{part1} together with properties of the Dunkl heat and the Dunkl Poisson kernels for studying  inhomogeneous Lipschitz spaces associated with the Dunkl operators on the Euclidean space $\mathbb R^N$. 
 
} 
\part{Semigroup approach to \texorpdfstring{$\Lambda^\beta$}{Lambda}-spaces}
\label{part1}

\section{Analytic semigroups of linear operators}
\subsection{Analytic semigroups.} In the present section we collect facts concerning holomorphic (analytic) semigroups of operators on Banach spaces (see e.g., \cite{Davies}, \cite{Pazy}). 

Let $\{\mathcal{T}_t\}_{t \geq 0}$  be a  strongly continuous semigroup of linear operators on a Banach space ${(X,\|\cdot\|)}$, which, for certain $0<\delta<\pi/2$,  has an extension to a uniformly bounded  holomorphic semigroup $\{\mathcal{T}_z\}_{z\in {\boldsymbol \Delta}_\delta}$ in  the  sector 
$$ {\boldsymbol \Delta}_\delta=\{z\in\mathbb C: |\text{arg}\, z|<\delta\}$$ around the positive axis. 

Let ${(X^*,\|\cdot\|_{X^*})}$ be the dual space to { the Banach space $(X, \|\cdot\|)$} and let $\mathcal{T}_t^*\in \mathcal L(X^*)$ denote the dual operators to $\mathcal T_t$.  Then $\{\mathcal{T}_t^*\}_{t \geq 0}$ has a unique extension to a uniformly bounded  holomorphic semigroup (which is in general not strongly continuous). It follows from the theory of analytic semigroups that if   $({\mathcal{A}},\mathfrak D({\mathcal{A}}))$ is the infinitesimal generator of $\{\mathcal{T}_t\}_{t \geq 0}$, then 
\begin{enumerate}[(1)]
    \item $\mathcal{T}_t(X)\subseteq \bigcap_{n \in \mathbb{N}} \mathfrak D({\mathcal{A}}^n)$ for all $t>0$, 
    \item the functions 
    $(0,\infty)\ni t\mapsto \mathcal{T}_t\in \mathcal L(X)$, \ $  (0,\infty)\ni t\mapsto \mathcal{T}_t^*\in \mathcal L(X^*) $  are differentiable (even holomorphic in ${\boldsymbol \Delta}_\delta$) and 
    \begin{equation} \frac{d}{dt} \mathcal{T}_t={\mathcal{A}}\mathcal{T}_t, \quad  \frac{d}{dt} \mathcal{T}^*_t= ({\mathcal{A}}\mathcal{T}_t)^*, 
    \end{equation}
    \begin{equation} \| {\mathcal{A}}\mathcal{T}_t\|= \| ({\mathcal{A}}\mathcal{T}_t)^*\|_{X^*}\leq Ct^{-1},
    \end{equation}
    \begin{equation}\label{eq:dt=A} {\mathcal{A}}\mathcal{T}_{t+s} = \mathcal{T}_t{\mathcal{A}}\mathcal{T}_s, \quad ({\mathcal{A}}\mathcal{T}_{t+s})^*=\mathcal{T}_t^*({\mathcal{A}}\mathcal{T}_s)^*.
    \end{equation}
    In particular, if $t=t_1+t_2+...+t_n$, $t_j>0$, then 
    \begin{equation}\label{eq:time_deriv}
    \frac{d^n}{dt^n}\mathcal{T}_t={\mathcal{A}}^n\mathcal{T}_t={\mathcal{A}}\mathcal{T}_{t_1}{\mathcal{A}}\mathcal{T}_{t_2}...{\mathcal{A}}\mathcal{T}_{t_n}, \quad \frac{d^n}{dt^n}\mathcal{T}^*_t=({\mathcal{A}}^n\mathcal{T}_t)^*=({\mathcal{A}}\mathcal{T}_{t_1})^*({\mathcal{A}}\mathcal{T}_{t_2})^*...({\mathcal{A}}\mathcal{T}_{t_n})^*.
    \end{equation}
\end{enumerate} 
Consequently, there is $C>0$ such that for all $n \in \mathbb{N}$ and $t>0$ we have
\begin{equation}\label{eq:trivial}
    \Big\|\frac{d^n}{dt^n}\mathcal{T}_t\Big\| \leq C^n t^{-n}, \quad \Big\|\frac{d^n}{dt^n}\mathcal{T}^*_t\Big\|_{X^*} \leq C^n t^{-n}. 
\end{equation}

By $\rho ({\mathcal{A}})$, we denote the resolvent set of ${\mathcal{A}}$, that is, the set of all $\lambda\in\mathbb C$ such that the operator $\lambda I -{\mathcal{A}}: \mathfrak D({\mathcal{A}})\to X$ is one to one and onto, and its inverse, denoted by $R(\lambda:{\mathcal{A}})$, is bounded on $X$. The resolvent set $\rho({\mathcal{A}})$ is open and the mapping $\rho({\mathcal{A}})\ni\lambda \mapsto R(\lambda:{\mathcal{A}})\in \mathcal L(X)$ is holomorphic, so is $R(\lambda:{\mathcal{A}})^*$.  Moreover, for all $\lambda, \mu \in \rho (\mathcal{A})$, 
\begin{equation}\begin{split}\label{eq:res_prop1} 
& R(\lambda:{\mathcal{A}})-R(\mu:{\mathcal{A}})=(\mu-\lambda)R(\lambda:{\mathcal{A}})R(\mu:{\mathcal{A}}), \quad \frac{d}{d\lambda} R(\lambda:{\mathcal{A}})=-R(\lambda:{\mathcal{A}})^2,\\
& R(\lambda:{\mathcal{A}})^*-R(\mu:{\mathcal{A}})^*=(\mu-\lambda)R(\lambda:{\mathcal{A}})^*R(\mu:{\mathcal{A}})^*, \quad \frac{d}{d\lambda} R(\lambda:{\mathcal{A}})^*=-R(\lambda:{\mathcal{A}})^{*2}.
\end{split}\end{equation}

If $\{\mathcal{T}_t\}_{t \geq 0}$ is a holomorphic semigroup, uniformly bounded in ${\boldsymbol \Delta}_\delta$, then 
\begin{equation}\label{eq:sihma_del} \Sigma_\delta :=\Big\{\lambda \in \mathbb{C}: |\text{\rm arg}\, \lambda|<\frac{\pi}{2} +\delta\Big\} \subseteq \rho ({\mathcal{A}}),
\end{equation} 
\begin{equation}\label{eq:R_lambda} \| R(\lambda :{\mathcal{A}})\|\leq \frac{C'}{|\lambda|}, \quad \text{\rm for } \lambda\in \Sigma_\delta. 
\end{equation}

\subsection{$\Lambda^\beta_{\mathcal A}$-spaces - basic properties}\label{Lambda_spaces}
Let $\{\mathcal{T}_t\}_{t \geq 0}$ be a semigroup  of linear operators generated by $\mathcal A$ which has extension to a uniformly bounded holomorphic semigroup in a sector ${\boldsymbol \Delta}_\delta$. In this section we state elementary properties of the space $\Lambda^\beta_{\mathcal A}$ (see Definition \ref{def:Lambda}).

The following {  easily proved} proposition asserts that, as in the classical case (cf. \cite[Chapter V, Lemma 5]{Stein_singular}), the integer $m$ in  Definition \ref{def:Lambda} can be replaced by any integer $n>\beta$. 

\begin{proposition}\label{pro:any_m}
    Let $x^*\in X^*$. Fix  $\beta >0$. If $n>m>\beta$  are two integers, then the following two conditions 
    \begin{equation}\label{eq:condition_m_n}
        \Big\|\frac{d^m}{dt^m} \mathcal{T}_t^* x^*\Big\|_{X^*}\leq C_mt^{\beta-m}, \quad    \Big\|\frac{d^n}{dt^n}  \mathcal{T}_t^* x^*\Big\|_{X^*}\leq C_nt^{\beta-n},
    \end{equation}
    are equivalent. Moreover, there is  $C>1$ (which depends on $m$ and $n$ and it is independent of $x^*\in X^*$) such that the smallest constant $C_m$ and $C_n$ holding in the above inequalities satisfy: 
    $$ C^{-1}C_m\leq C_n\leq C C_m.$$
\end{proposition}

\begin{remark}\label{rem:small_t}\normalfont 
    It is worth to emphasis that, as in the classical case, the  estimates \eqref{eq:condition_m_n} are of interest  only for $0<t\leq 1$, because, for $t >1$, thanks to \eqref{eq:trivial}, the following better estimates always hold, namely 
    $$ \Big\|\frac{d^n}{dt^n} \mathcal{T}_t^* x^*\Big\|_{X^*} \leq C^n t^{-n} \|x^*\|_{X^*}.$$
\end{remark}

Remark \ref{rem:small_t} together with Proposition \ref{pro:any_m} imply that 
\begin{equation}
    \label{eq:inclusion_trivial}
    \Lambda_{\mathcal{A}}^{\beta_1}\subseteq \Lambda_{\mathcal{A}}^{\beta_2} \quad \text{and} \quad \| x^*\|_{\Lambda^{\beta_2}_{\mathcal{A}}} \leq C_{\beta_1,\beta_2} \| x^*\|_{\Lambda^{\beta_1}_{\mathcal{A}}} \quad \text{for } 
    \beta_1\geq \beta_2>0. 
\end{equation}
\begin{lemma}
   \label{lem:converge} Suppose $\beta>0$. If $x^*\in \Lambda^\beta_{\mathcal A}$, then 
   \begin{equation}
       \label{eq:smal_derv}
       \Big\|\frac{d^n}{dt^n} \mathcal T_t^*x^*\Big\|_{X^*} \leq C_n \| x^*\|_{\Lambda^\beta_{\mathcal A}} \quad \text{for } \ n<\beta; 
   \end{equation}
   \begin{equation}\label{eq:strong_conv}
       \lim_{t\to 0} \| x^*-\mathcal T^*_tx^*\|_{X^*}=0.
   \end{equation}
\end{lemma}
\begin{proof} The proof of \eqref{eq:smal_derv} follows from integration of \eqref{eq:alpha_decay}. To prove \eqref{eq:strong_conv}, 
    there is no loss of generality if we assume that $0<\beta<1$. Let  $0<t_1<t_2<1$.  We write 
    \begin{equation*}
        \begin{split}
           \| \mathcal T_{t_2}x^*-\mathcal T^*_{t_1}x^*\|_{X^*} &= 
    \Big\| \int_{t_1}^{t_2} \frac{d}{ds} \mathcal T_s^* x^*\, ds\Big\|_{X^*} 
    \leq C\int_{t_1}^{t_2} s^{\beta-1}\, ds=C'(t_2^\beta -t_1^\beta). 
        \end{split}
    \end{equation*}
    Hence, $\mathcal T_t^*x^*$ satisfies the  Cauchy condition,  as $t\to 0$. Let $x_0^*=\lim_{t\to 0}\mathcal T_t^*x^*$ in the $\| \cdot\|_{X^*}$-norm. Then 
    \begin{equation*}
        \langle x_0^*,x\rangle =\lim_{t\to 0} \langle \mathcal T_t^*x^*,x\rangle = \lim_{t\to 0} \langle x^*, \mathcal T_tx\rangle =\langle x^*,x\rangle, 
    \end{equation*}
    because $\{\mathcal T_t\}_{t \geq 0}$ is strongly continuous. So $x_0^*=x^*$.
\end{proof}

\begin{proof}[Proof of Theorem \ref{theo:A_Lambda}] 
Suppose that $x^*\in X^*$ is such that $y^*:={\mathcal{A}}^*x^*\in \Lambda^\beta_{\mathcal A}$. Let $m>\beta+1$. Then $m-1> \beta$ and 
\begin{equation}\label{eq:yinLambda} \Big\| \frac{d^{m-1}}{dt^{m-1}} \mathcal{T}_t^*y^*\Big\|_{X^*}\leq t^{\beta -(m-1)} \|y^*\|_{\Lambda^\beta_{\mathcal A}}.
\end{equation} 
Since $\{\mathcal{T}_t\}_{t \geq 0}$ is a strongly continuous and uniformly bounded semigroup of linear operators, using \eqref{eq:time_deriv} and~\eqref{eq:strong_conv}, we get 
\begin{equation}
    \begin{split}
        \Big\| \frac{d^m}{dt^m} \mathcal{T}^*_tx^*\Big\|_{X^*} & \leq \sup_{s>0,\, \|x\|\leq 1} \Big|\Big\langle \frac{d^m}{dt^m} \mathcal{T}^*_tx^*, \mathcal{T}_sx\Big\rangle\Big|\\
        &=  \sup_{s>0,\, \|x\|\leq 1} |\langle x^*, {\mathcal{A}}\mathcal{T}_{t/m} ({\mathcal{A}} \mathcal{T}_{t/m})^{m-1} \mathcal{T}_sx\rangle|\\
        &= \sup_{s>0,\, \|x\|\leq 1} |\langle y^*, \mathcal{T}_{t/m} ({\mathcal{A}} \mathcal{T}_{t/m})^{m-1} \mathcal{T}_sx\rangle|\\
        &=\sup_{s>0,\, \|x\|\leq 1} \Big|\Big\langle \frac{d^{m-1}}{dt^{m-1}} \mathcal{T}_t^*y^*,   \mathcal{T}_sx\Big\rangle\Big| \leq C \|y^*\|_{\Lambda^\beta_{\mathcal A}} t^{(\beta+1)-m},
    \end{split}
\end{equation}
where in the last inequality we have used \eqref{eq:yinLambda}. 

We now turn to prove the converse implication. Suppose that $x^*\in \Lambda^{\beta+1}_{\mathcal{A}}$. Let $m$ be an integer, such that $1\leq m-1\leq \beta +1<m$. We start by proving that ${\mathcal{A}}^*x^*$ exists in the mild sense. Let $u(t):=\mathcal{T}_t^*x^*$, $v(t):=u'(t)=({\mathcal{A}}\mathcal{T}_t)^*x^*$. Then, by our assumption, 
\begin{equation}\label{eq:u(m)} 
\| v^{(m-1)}(t)\|_{X^*}= \| u^{(m)}(t)\|_{X^*} \leq  \|x^*\|_{\Lambda^{\beta+1}_{\mathcal A}} t^{(\beta+1)-m}.
\end{equation}
Consider the Taylor expansion of the function $v(t)$ around the point $t_0=1$: 
\begin{equation}\label{eq:Taylor_v}
v(t)=\sum_{\ell=0}^{m-2} \frac{1}{\ell !} v^{(\ell)} (1) (t-1)^\ell +\int_1^t \frac{(t-s)^{m-2}}{(m-2)!} v^{(m-1)} (s)\, ds . 
\end{equation} 
It follows from \eqref{eq:u(m)} that $v(t)$ converges in the $X^*$-norm to a vector $y^*\in X^*$, as $t$ tends to 0,  and 
$\|y^*\|_{X^*}\leq C \|x^*\|_{\Lambda^{\beta+1}_{\mathcal A}}$.   To see that $y^*={\mathcal{A}}^*x^*$, we write 
\begin{equation*}
    \begin{split}
        \langle y^*, \mathcal{T}_sx\rangle  = \lim_{t\to 0} \langle v(t),\mathcal{T}_sx\rangle 
        =\lim_{t\to 0} \langle ({\mathcal{A}}\mathcal{T}_t)^* x^* , \mathcal{T}_sx\rangle 
        = \lim_{t\to 0} \langle x^*, {\mathcal{A}}\mathcal{T}_t\mathcal{T}_sx\rangle = \langle x^*, {\mathcal{A}}\mathcal{T}_sx\rangle, 
    \end{split}
\end{equation*}
{where in the last equality we have used~\eqref{eq:dt=A} and the strong continuity of $\{\mathcal T_t\}_{t \geq 0}$. }

It suffices to verify that $y^*\in \Lambda_{\mathcal{A}}^{\beta}$. To this end, we recall  that $\beta<m-1$ and, for $t>0$, applying \eqref{eq:time_deriv}, we obtain   
\begin{equation*}
    \begin{split}
        \Big\| \frac{d^{m-1}}{dt^{m-1}} \mathcal T^*_t y^*\Big\|_{X^*} &=\lim_{s\to 0} \| (({\mathcal{A}}\mathcal{T}_{t/(m-1)})^*)^{m-1} v(s)\|_{X^*}   =\lim_{s\to 0} \| (({\mathcal{A}}\mathcal{T}_{t/(m-1)})^*)^{m-1} ({\mathcal{A}}\mathcal{T}_s)^* x^*\|_{X^*}\\
        &= \lim_{s\to 0} \Big\| \frac{d^{m}}{d\tau^m} \mathcal{T}_\tau x^*_{\big| \tau=t+s} \Big\|_{X^*} \leq  t^{(\beta+1)-m}\|x^*\|_{\Lambda^{\beta+1}_{\mathcal A}} = t^{\beta -(m-1)} \|x^*\|_{\Lambda^{\beta+1}_{\mathcal A}}. 
    \end{split}
\end{equation*}
\end{proof}

\subsection{Bessel  potentials} 

For $\gamma>0$, we define the potential operator 
\begin{equation}\label{eq:Bessel} \mathcal J^\gamma:=(I-{\mathcal{A}})^{-\gamma}=\Gamma (\gamma)^{-1} \int_0^\infty t^\gamma e^{-t} \mathcal{T}_t\,\frac{dt}{t} 
\end{equation}
and its conjugate operator 
\begin{equation}\label{eq:Bessel*} \mathcal J^{\gamma *}:=((I-{\mathcal{A}})^{-\gamma})^*= \Gamma (\gamma)^{-1} \int_0^\infty t^\gamma e^{-t} \mathcal{T}_t^*\,\frac{dt}{t}.
\end{equation}
Let us note that the integrals converge in the norm operator topology and define bounded operators on $X$ and $X^*$ respectively.  
Moreover, {for $\gamma_1,\gamma_2>0$,}
\begin{equation}\label{eq:composition} \mathcal J^{\gamma_1} \mathcal J^{\gamma_2}=\mathcal J^{\gamma_1+\gamma_2}, \quad \mathcal J^{\gamma_1\, *} \mathcal J^{\gamma_2\, *}=\mathcal J^{(\gamma_1+\gamma_2)\, *}.  
\end{equation}

\begin{proof}[Proof of Theorem \ref{theo:Bessel}] 

We start by proving that for any $\beta>0$, the operator $\mathcal J^{\gamma *}$ is  bounded  from $\Lambda_{\mathcal{A}}^{\beta} $ to $\Lambda_{\mathcal{A}}^{\beta+\gamma} $. Clearly, for $x^*\in\Lambda_{\mathcal{A}}^\beta$, we have 
\begin{equation}  \|\mathcal J^{\gamma\, *} x^*\|_{X^*}\leq C \|x^*\|_{X^*}\leq C' \| x^*\|_{\Lambda_{\mathcal{A}}^\beta}.
\end{equation}
To verify that $y^*:=J^{\gamma \, *} x^*$ belongs to $\Lambda^{\beta+\gamma}_{\mathcal{A}}$, and $\|y^*\|_{\Lambda_{\mathcal{A}}^{\beta+\gamma}}\leq C \| x^*\|_{\Lambda_{\mathcal{A}}^\beta}$,   we fix $n>\beta+\gamma$ and  $0<t\leq 1$. {Then, using \eqref{eq:time_deriv} and our assumption,  we obtain }  
   \begin{equation}
       \begin{split}
           \Big\| \frac{d^n}{dt^n} \mathcal{T}_t^*y^*\Big\|_{X^*} & = \Gamma(\gamma)^{-1}\Big\| \int_0^\infty s^{\gamma}e^{-s} (({\mathcal{A}}\mathcal{T}_{t/n})^{*})^n \mathcal{T}_s^*x^*\frac{ds}{s} \Big\|_{X^*}\\
           & = \Gamma(\gamma)^{-1}\Big\| \int_0^\infty s^{\gamma}e^{-s} (({\mathcal{A}}\mathcal{T}_{(t+s)/n})^{*})^n x^*\frac{ds}{s} \Big\|_{X^*}\\
           &\leq C \|x^*\|_{\Lambda^\beta_{\mathcal A}}\int_0^t s^{\gamma} e^{-s} t^{\beta -n} \frac{ds}{s} + 
           C\|x^*\|_{\Lambda^\beta_{\mathcal A}}\int_t^\infty  s^{\gamma} e^{-s} s^{\beta -n} \frac{ds} {s}\\
          & \leq C' \|x^*\|_{\Lambda^\beta_{\mathcal A}}t^{\beta +\gamma -n}. 
       \end{split}
   \end{equation} 
   {Thus we have shown that {$\mathcal J^{\gamma\, *}$} is a continuous linear transformation of $\Lambda^\beta_{\mathcal A}$ into  $\Lambda^{\beta+\gamma}_{\mathcal A}$.  } 
   
   According to the Banach closed graph theorem,  it remains to prove that the mapping 
   $ \mathcal J^{\gamma\, * } : \Lambda_{\mathcal{A}}^\beta \to \Lambda_{\mathcal{A}}^{\beta+\gamma} $  is injective and onto. Thanks to \eqref{eq:composition},  by standard functional analysis arguments, it suffices to prove this for $\gamma =1$, {cf. \cite[Chapter V, Section 4.4]{Stein_singular} for the proof in the case of the classical Lipschitz spaces.}  
   Since $R(1:{\mathcal{A}})(X)=\mathfrak D({\mathcal{A}})$, which is a dense space in $X$,  we conclude that $\mathcal J^{1\,*} = R(1:{\mathcal{A}})^*$ is injective. To verify that $R(1:{\mathcal{A}})^*$ is onto, consider  $x^*\in \Lambda_{\mathcal{A}}^{\beta+1}$. Then $y^*=x^*-{\mathcal{A}}^*x^*$ belongs to $\Lambda_{\mathcal{A}}^\beta$, where ${\mathcal{A}}^*x^*$ is understood in the mild sense, that is, $\langle y^*, \mathcal{T}_tx\rangle=\langle x^*, (I-{\mathcal{A}})\mathcal{T}_tx\rangle$.  Now,
   \begin{equation}\begin{split}
       \langle R(1:{\mathcal{A}})^* y^*, \mathcal{T}_tx\rangle &= \langle y^* , R(1:{\mathcal{A}})\mathcal{T}_tx\rangle = \langle y^* , \mathcal{T}_tR(1:{\mathcal{A}})x\rangle  \\&= \langle x^* , (I-{\mathcal{A}}) \mathcal{T}_t R(1: {\mathcal{A}})x\rangle = \langle x^*,\mathcal{T}_tx\rangle,
  \end{split} \end{equation} and, consequently, $R(1:{\mathcal{A}})^*y^*=x^*$. 
\end{proof}

{
\section{Subordinate semigroup} 

In this section, for a holomorphic uniformly bounded semigroup $\{\mathcal{T}_t\}_{t \geq 0}=\{e^{t{\mathcal{A}}}\}_{t \geq 0}$ on a Banach space $X$, we consider the  subordinate semigroup $\{{\mathcal{P}}_t\}_{t \geq 0}$ defined by \eqref{eq:P_sub}.

The integral in~\eqref{eq:P_sub} is convergent in the norm operator topology $\mathcal L(X)$ and it defines a strongly continuous uniformly bounded (in any sector ${\boldsymbol \Delta}_\delta$, $0<\delta<\pi/2$) holomorphic semigroup on $X$. Our aim is to prove Theorem \ref{theo:Lamba=Lambda}. The relations: 
\begin{equation}\label{eq:easier}
    \Lambda_{{\mathcal{A}}}^\beta\subseteq \Lambda_{-\sqrt{-{\mathcal{A}}}}^{2\beta}  \quad  \text{with } 
 \| x^*\|_{\Lambda_{-\sqrt{-{\mathcal{A}}}}^{2\beta}}\leq C \| x^*\|_{\Lambda_{\mathcal{A}}^\beta} \quad  \text{ for all }x^* \in \Lambda^{\beta}_{\mathcal A} 
\end{equation}
will be proved by utilizing \eqref{eq:P_sub}. While for the converse, namely for  
\begin{equation}\label{eq:noteasy}
     \Lambda_{-\sqrt{-{\mathcal{A}}}}^{2\beta}   \subseteq  \Lambda_{{\mathcal{A}}}^\beta \quad  \text{with } 
 C^{-1} \| x^*\|_{\Lambda_{\mathcal{A}}^\beta }
 \leq  \| x^*\|_{\Lambda_{-\sqrt{-{\mathcal{A}}}}^{2\beta} }  \quad \text{ for all }x^* \in \Lambda^{2\beta}_{-\sqrt{-\mathcal A}},  
 \end{equation}
we shall apply a holomorphic functional calculi. For the reader who is not familiar with the methods, we provide details in Section \ref{sec:holomorphic}.

\subsection{Inclusion $ \Lambda_{{\mathcal{A}}}^\beta\subseteq \Lambda_{-\sqrt{-{\mathcal{A}}}}^{2\beta} $}
We start with the following lemma. 

\begin{lemma}\label{lem:phi_sub}
   Suppose that  $(0,\infty)\ni t\mapsto \phi (t)$ is a $C^\infty$-bounded  function taking values in a Banach space such that for any $n \in \mathbb{N}$ there is $C_n>0$ such that for all $t>0$ we have
   \begin{equation}\label{eq:con_g} \|\phi^{(n)}(t)\|\leq C_nt^{-n}.  
   \end{equation}
   Then the function 
   $$ (0,\infty)\ni t\mapsto \psi (t):=\int_0^\infty e^{-u} \phi\Big(\frac{t^2}{4u}\Big)\frac{du}{\sqrt{u}}$$
   is $C^\infty$ and  satisfies 

 $$ \psi^{(2n+1)}(t)=(-1)^n\int_0^\infty e^{-u}\frac{t}{2u}\phi^{(n+1)}\Big(\frac{t^2}{4u}\Big)\, \frac{du}{\sqrt{u}},$$
   
   $$ \psi^{(2n)}(t)=(-1)^n\int_0^\infty e^{-u}\phi^{(n)}\Big(\frac{t^2}{4u}\Big)\, \frac{du}{\sqrt{u}}.$$
\end{lemma}

\begin{proof} 
Clearly, 
\begin{equation*}
    \psi'(t)=\int_0^\infty e^{-u} \frac{t}{2u} \phi'\Big(\frac{t^2}{4u}\Big) \frac{du}{\sqrt {u}}. 
\end{equation*}
  Let us note that for any smooth function $(0,\infty)\ni t \mapsto g(t)$, we have 
\begin{equation}\label{eq:t=u}
     \frac{d}{dt}[g(t^2/(4u))]=-\frac{2u}{t} \frac{d}{du}[g(t^2/(4u))].
\end{equation} 
Hence, 
\begin{equation*}
    \begin{split}
        \psi''(t)& =\int_0^\infty e^{-u} \frac{1}{2u} \phi'\Big(\frac{t^2}{4u}\Big) \frac{du}{\sqrt{u}} + \int_0^\infty e^{-u} \frac{t}{2u} \frac{d}{dt} \Big[\phi'\Big(\frac{t^2}{4u}\Big)\Big]\frac{du}{\sqrt{u}}\\
        &= \int_0^\infty e^{-u} \frac{1}{2u} \phi'\Big(\frac{t^2}{4u}\Big) \frac{du}{\sqrt{u}} - \int_0^\infty e^{-u}  \frac{d}{du} \Big[\phi'\Big(\frac{t^2}{4u}\Big)\Big]\frac{du}{\sqrt{u}},\\
    \end{split}
\end{equation*}
where in the last equation we have used \eqref{eq:t=u}. Integrating by parts and using \eqref{eq:con_g}, we obtain 
\begin{equation*}
    \begin{split}
        \psi''(t) 
        &= \int_0^\infty e^{-u} \frac{1}{2u} \phi'\Big(\frac{t^2}{4u}\Big) \frac{du}{\sqrt{u}} + \int_0^\infty   \frac{d}{du} \Big[ e^{-u} u^{-1/2}\Big]  \phi'\Big(\frac{t^2}{4u}\Big)\, du 
        =- \int_0^\infty e^{-u} \phi'\Big(\frac{t^2}{4u} \Big)\frac{du}{\sqrt{u}}. 
    \end{split}
\end{equation*} 
The proof for the  higher order derivatives goes by iterating  the argument presented  above. 
\end{proof}

\begin{proof}[Proof of the inclusion  $ \Lambda_{{\mathcal{A}}}^\beta\subseteq \Lambda_{-\sqrt{-{\mathcal{A}}}}^{2\beta} $]  Assume that $x^*\in \Lambda^\beta_{\mathcal{A}}$. Consider the functions  
$$\phi(t)={\Gamma(1/2)^{-1}}\mathcal{T}_t^*x^* \quad \text{and} \quad  \psi(t)={\mathcal{P}}_t^*x^*=\int_0^\infty e^{-u} \phi\Big(\frac{t^2}{u}\Big)\frac{du}{\sqrt{u}}$$ taking values in $X^*$. 
 If $n>\beta$ then, by Proposition~\ref{pro:any_m}, $\| \phi^{(n)}(s)\|\leq s^{\beta -n} \| x^*\|_{\Lambda_{\mathcal{A}}^\beta} $. 
Thus, applying Lemma \ref{lem:phi_sub}, we get 
 \begin{equation*} \begin{split} \| \psi^{(2n)}(t)\|_{X^*} & \leq \int_0^\infty e^{-u} \| \phi^{(n)}(t^2/4u)\|_{X^*} \frac{du}{\sqrt{u}}\\
 &\leq  \int_0^\infty e^{-u} \| x^*\|_{\Lambda_{\mathcal{A}}^\beta} \Big(\frac{t^2}{4u}\Big)^{\beta-n} \frac{du}{\sqrt{u}}\leq C t^{2\beta - 2n} \| x^*\|_{\Lambda^\beta_{\mathcal{A}}},   
\end{split} \end{equation*}
 which proves the inclusion and the second inequality in \eqref{eq:Lambda=Lambda}.  \end{proof} 
 }
\section{Holomorphic calculi}\label{sec:holomorphic}

This section is devoted for proving the inclusion $\Lambda_{-\sqrt{-{\mathcal{A}}}}^{2\beta}\subseteq \Lambda^\beta_{\mathcal{A}}$. To this end we utilize  holomorphic functional calculi for generators of semigroups. For the reader who is not familiar with the subject we provide all details. We refer to \cite{Haase} and references therein for more results. 

\subsection{Admissible holomorphic functions in $\mathbb{C}_{-}=\{z \in \mathbb{C}\::\: \text{\rm Re}(z)<0\}$}

\begin{definition}\label{def:admissible}\normalfont 
    We say that $f \in H_a^{\infty}$ {(space of admissible holomorphic functions)} if it satisfies the following conditions:
    \begin{enumerate}[(A)]
        \item{ $f$ is bounded and holomorphic on $\mathbb{C}_{-}=\{z \in \mathbb{C}\::\: {\rm Re}(z)<0\}$;}\label{bounded_holo}
        \item{for all $x \leq 0$ and $\frac{\pi}{2}<\theta<\pi$, we have
        \begin{equation}
            \int_0^\infty |f(x+re^{\pm i\theta})|\,dr<\infty.
        \end{equation}}\label{bounded_line}
    \end{enumerate}
\end{definition}

It is easy to check that $f_1,f_2 \in H^\infty_a$, where $f_1(z)=e^{z}$, $f_2(z)=e^{-\sqrt{-z}}$. {Here $\sqrt{z}$ denotes the branch of the square root defined on $\mathbb C\setminus \mathbb R_{-}$. }

Let $\frac{\pi}{2}<\theta<\pi$. We define the parameterized  path:

\begin{equation*}
    (-\infty,\infty) \ni t\mapsto \Gamma_\theta (t)=
    \begin{cases}
        |t|e^{-i\theta} \quad &\text{if } t\leq 0,\\
        |t|e^{i\theta} \quad &\text{if } t>0.
    \end{cases}
\end{equation*}

Then, we define the right and the left sides of $\Gamma_\theta$: 
\begin{equation}
    \Gamma_{\theta}^{+}:=\{\lambda+x\::\: \lambda \in \Gamma_\theta, \ x >0\},  \quad \Gamma_{\theta}^{-}:=\{\lambda+x\::\: \lambda \in \Gamma_\theta, \ x <0\}.
\end{equation}

{For $\varepsilon >0$ we define a modification $\Gamma_{\theta,\, \varepsilon}$ of the path $\Gamma_{\theta}$ by replacing its piece which corresponds to the two intervals for parameters $0\leq |t|\leq \varepsilon$ by the part of the circle $\varepsilon e^{i\omega}$, $\theta\leq |\omega|\leq \pi$, formally,
\begin{equation*}
    (-\infty,\infty) \ni t\mapsto \Gamma_{\theta,\varepsilon} (t)=
    \begin{cases}
        |t|e^{-i\theta} \quad &\text{if } t\leq -\varepsilon,\\
         \varepsilon e^{-it(\pi-\theta)/\varepsilon-i\pi} \quad &\text{if } -\varepsilon<t \leq \varepsilon,\\
        |t|e^{i\theta} \quad &\text{ if } t>\varepsilon.
    \end{cases}
\end{equation*}
 Similarly, 
\begin{equation}\label{eq:Gamma_pm}
    \Gamma_{\theta,\varepsilon}^{+}=\{\lambda+x\::\: \lambda \in \Gamma_{\theta,\varepsilon},  \ \ x >0, \}, \quad  \Gamma_{\theta,\varepsilon}^{-}=\{\lambda+x\::\: \lambda \in \Gamma_{\theta, \varepsilon},  \ x <0\}.
\end{equation}

 To unify our notation, we set $\Gamma_{\theta, \varepsilon}:=\Gamma_\theta$, if $\varepsilon =0$.

\begin{lemma}[cf. {\cite[Lemma 8.2]{Haase}}]\label{lem:basic_lemma}
    For $f \in H^\infty_a$ and $\lambda \not\in \Gamma_{ \theta, \varepsilon}$, $\varepsilon \geq 0$, $\pi/2<\theta<\pi$,  let 
    \begin{equation}\label{eq:f^tilde}
        \widetilde{f}(\lambda)=\frac{1}{2\pi i}\int_{\Gamma_{\theta,\varepsilon}}\frac{f(z)}{z-\lambda}\,dz.
    \end{equation}
    Then $\widetilde{f}$ is holomorphic on $\mathbb{C} \setminus \Gamma_{\theta,\varepsilon}$ and
    \begin{equation}
         \widetilde{f}(\lambda)=\begin{cases}
            f(\lambda) & \text{\rm if }\lambda \in \Gamma_{\theta,\varepsilon}^{-},\\
            0  & \text{\rm if }\lambda \in \Gamma_{\theta,\varepsilon}^{+}.
        \end{cases}
    \end{equation}
\end{lemma}

\begin{proof} We provide the proof for $\varepsilon >0$. The proof for  $\varepsilon =0$ is obtained by taking $\varepsilon\to 0$.

    \textbf{Case $\lambda \in \Gamma_{\theta,\varepsilon}^{+}$}. { It follows from \eqref{eq:f^tilde} that } 
    \begin{align*}
        \widetilde{f'}(\lambda)=\frac{1}{2\pi i}\int_{\Gamma_{\theta, \varepsilon}}\frac{f(z)}{(z-\lambda)^2}\,dz.
    \end{align*}
    We will show that $\widetilde{f}'(\lambda)=0$. To this end we fix $\lambda\in \Gamma_{\theta,\varepsilon}^+$. For $x<0$ and $R>2\varepsilon $, let us define the closed path $\boldsymbol{\gamma}^{x,R}=\bigcup_{j=1}^{4}\boldsymbol{\gamma}_j^{x,R}$, oriented counterclockwise, which consists of two curves:  
    \begin{equation*}
        \boldsymbol{\gamma}_1^{x,R}=\{\Gamma_{\theta,\varepsilon}(t)\::\: t\in [-R,R]\}, \quad \boldsymbol{\gamma}_2^{x,R}=\boldsymbol{\gamma}_1{^{x,R}}+x, \end{equation*}
        and two line segments:         
    \begin{equation*}\begin{split}
        &\boldsymbol{\gamma}_3^{x,R}=\{(Re^{i \theta}+x)(1-s) +Re^{i\theta}s: s\in [0,1]\}, \\  &\boldsymbol{\gamma}_4^{x,R}= \{(Re^{-i \theta}+x)(1-s) +Re^{-i\theta}s: s\in [0,1]\}.
    \end{split} \end{equation*}
    By the fact that the function $z \mapsto f(z)(z-\lambda)^{-2}$ is holomorphic in a open neighborhood of the compact region bounded by  $\boldsymbol{\gamma}^{x,R}$,  we have  
    \begin{equation}\label{eq:full_holo}
        \oint_{\boldsymbol{\gamma}^{x,R}}\frac{f(z)}{(z-\lambda)^2}\,dz=\sum_{j=1}^{4} \int_{\boldsymbol{\gamma}_j^{x,R}}\frac{f(z)}{(z-\lambda)^2}\,dz=0.
    \end{equation}
    From  the condition~\eqref{bounded_holo} we conclude that for each (fixed) $x<0$, we have 
    \begin{equation}\label{eq:34_holo}
        \lim_{R \to \infty}\int_{\boldsymbol{\gamma}_j^{x,R}}\frac{f(z)}{(z-\lambda)^2}\,dz=0 \quad \text{ for } j=3,4, 
    \end{equation}
    and, consequently, by \eqref{eq:full_holo}, for each (fixed) $x<0$, one has  
    \begin{equation}\label{eq:sum2=0}
        \int_{\Gamma_{\theta,\varepsilon}} \frac{f(z)}{(z-\lambda)^2}\, dz - \int_{\Gamma_{\theta,\varepsilon} +x}\frac{f(z)}{(z-\lambda)^2}\, dz =0.
    \end{equation}
     Next, applying~\eqref{bounded_holo}, we obtain 
    \begin{equation}\label{eq:2_holo}
        \lim_{x \to -\infty}\int_{\Gamma_{\theta,\varepsilon}+x} \frac{f(z)}{(z-\lambda)^2}\,dz=0.
    \end{equation}
    Hence from \eqref{eq:sum2=0} and \eqref{eq:2_holo} we conclude that  $\widetilde{f}'(\lambda)=0$.
    Therefore, $\widetilde{f}$ is constant on  $\Gamma_{\theta,\varepsilon}^{+}$. To determine its value, thanks to the condition~\eqref{bounded_line}, we get
    \begin{equation*}
        \lim_{\lambda\in\mathbb R,\ \lambda  \to \infty}\int_{\Gamma_{\theta,\varepsilon}}\frac{f(z)}{z-\lambda}\,dz=0,
    \end{equation*}
    so $\widetilde f(\lambda)=0$ on $\Gamma_{\theta,\varepsilon} ^+$. 

    By the same argument we also conclude that 
    \begin{equation}\label{eq:right_zero}
        \frac{1}{2\pi i}\int_{\Gamma_{\theta,\varepsilon}+x}\frac{f(z)}{z-\lambda}\,dz=0 \quad \text{ for all }x<0 \quad { \text{and} \quad \lambda\in (\Gamma_{\theta,\varepsilon}+x)^+ }.
    \end{equation}
     \textbf{Case $\lambda \in \Gamma_{\theta,\varepsilon}^{-}$}. Let as consider $\boldsymbol{\gamma}^{x,R}$ as above {such that $\lambda$ stays inside the bounded region with the boundary $\boldsymbol{\gamma}^{x,R}$.} The Cauchy integral formula asserts that 
     \begin{equation}\label{eq:full_holo_1}
        f(\lambda)=\frac{1}{2\pi i} \oint_{\boldsymbol{\gamma}^{x,R}}\,\frac{f(z)}{z-\lambda}\,dz=\sum_{j=1}^{4} \frac{1}{2\pi i} \int_{\boldsymbol{\gamma}_j^{x,R}}\, \frac{f(z)}{z-\lambda}\,dz.
    \end{equation}
    From the condition~\eqref{bounded_holo}, we obtain 
    \begin{equation}\label{eq:34_holo_1}
        \lim_{R \to \infty}\int_{\boldsymbol{\gamma}_j^{x,R}}\, \frac{f(z)}{z-\lambda }\,dz=0 \quad \text{for} \quad j=3,4.
    \end{equation}
     Hence, taking $R \to \infty$ in~\eqref{eq:full_holo_1}, we obtain 
    \begin{equation}\label{eq:full_full}
        f(\lambda)=\frac{1}{2\pi i}\int_{\Gamma_{\theta,\varepsilon}}\frac{f(z)}{z-\lambda}\,dz-\frac{1}{2\pi i}\int_{\Gamma_{\theta,\varepsilon} +x}\, \frac{f(z)}{z-\lambda}\,dz.
    \end{equation}
    But, for $x<0$, $|x|$ large enough (i.e. $\lambda$ stays on the right side of the curve $\Gamma_{\theta,\varepsilon} +x$), { \eqref{eq:right_zero} gives }
    \begin{equation}\label{eq:int=0}
        \frac{1}{2\pi i}\int_{\Gamma_{\theta,\varepsilon}+x}\, \frac{f(z)}{z-\lambda}\,dz=0.
    \end{equation}
    Consequently,~\eqref{eq:full_full} { and \eqref{eq:int=0} lead   to $\widetilde f(\lambda)=f(\lambda)$ for $\lambda \in \Gamma_{\theta,\varepsilon}^-$. } 
\end{proof}

\subsection{Properties of resolvent}
{ In this subsection we assume that ${\mathcal{A}}$ is a closed operator on a Banach space such that 
 \begin{equation}\label{eq:res}
        \rho({\mathcal{A}}) \supseteq \Sigma_\delta \cup \{0\}=\{z \in \mathbb{C}\::\: |\text{arg}(z)|<\frac{\pi}{2}+\delta\} \cup \{0\}
    \end{equation}
    and there is $C>0$ such that for all $\lambda \in \Sigma_\delta$, $\lambda \neq 0$, we have
    \begin{equation}\label{eq:res_bound}
        \|R(\lambda:{\mathcal{A}})\| \leq \frac{C}{|\lambda|}.
    \end{equation}
Since $0\in \rho ({\mathcal{A}})$, there is $r>0$ such that $B(0,r)=\{z \in \mathbb{C}:|z|<r\}\subseteq \rho ({\mathcal{A}})$ and 
\begin{equation}\label{eq:reslovent}
    \sup_{z\in B(0,r)} \| R(z:{\mathcal{A}})\| <\infty. 
\end{equation}
\begin{lemma}\label{lem:resolvent_integral}
    Let $\pi<\theta<2\pi$ and { $\varepsilon\geq 0$ be such that $\Gamma_{\theta,\varepsilon} \subseteq \Sigma_\delta \cup B(0,r)$. Suppose  $z \in   \Gamma_{\theta,\varepsilon}^{+}$.} Then
    \begin{equation}\label{eq:resolvent_formula}
        R(z:{\mathcal{A}})=\frac{1}{2\pi i}\int_{\Gamma_{\theta,\varepsilon}}\frac{1}{z-\lambda}R(\lambda:{\mathcal{A}})\,d\lambda.
    \end{equation}
   {  Further, if $z_1,z_2\in \Gamma_{\theta,\varepsilon}^+$, then 
    \begin{equation}\label{eq:product_Cauchy}
        R(z_1:{\mathcal{A}})R(z_2:{\mathcal{A}})=\frac{1}{2\pi i} \int_{\Gamma_{\theta,\varepsilon}} \frac{1}{(z_1-\lambda)(z_2-\lambda) } R(\lambda:{\mathcal{A}})\, d\lambda.  
    \end{equation}} 

\end{lemma}

\begin{proof}
We start by proving ~\eqref{eq:resolvent_formula}. { Let $z\in\Gamma^{+}_{\theta, \varepsilon}$. } 
    Let us consider the following parameterized (closed) contour,  
    \begin{equation*}
    (-R,2R) \ni t\mapsto \Gamma_{\theta,\varepsilon,R} (t)=
    \begin{cases}
       \Gamma_{\theta,\varepsilon}(t) :=\boldsymbol{\gamma}_{R,1}(t)\quad &\text{if } -R<t \leq R
        ,\\
        Re^{-2it\theta/R+i3\theta}=: \boldsymbol{\gamma}_{R,2}(t) \quad &\text{if } R \leq t < 2R.
    \end{cases}
\end{equation*}
The contour $ \Gamma_{\theta,\varepsilon,R}$ is oriented  clockwise. Hence, taking $R$ large enough and using the  Cauchy formula, we get:
\begin{equation}\label{eq:closed_con}
\begin{split}
    R(z:{\mathcal{A}})&=-\frac{1}{2\pi i}\int_{\Gamma_{\theta, \varepsilon,R}}\frac{1}{\lambda -z}R(\lambda:{\mathcal{A}})\,d\lambda\\&=\frac{1}{2\pi i}\int_{\boldsymbol{\gamma}_{R,1}}\frac{1}{z-\lambda }R(\lambda:{\mathcal{A}})\,d\lambda+\frac{1}{2\pi i}\int_{\boldsymbol{\gamma}_{R,2}}\frac{1}{z-\lambda }R(\lambda:{\mathcal{A}})\,d\lambda.
\end{split}
\end{equation}
Let us note that thanks to~\eqref{eq:res_bound} we have
\begin{equation*}
\begin{split}
    \lim_{R \to \infty}\left|\frac{1}{2\pi i}\int_{\boldsymbol{\gamma}_{R,2}}\frac{1}{z-\lambda }R(\lambda:{\mathcal{A}})\,d\lambda\right| \leq \lim_{R \to \infty}\frac{C}{R}=0.
\end{split}
\end{equation*}
Moreover, by the definition of $\Gamma_{\theta,\varepsilon}$, we have
\begin{equation*}
    \lim_{R \to \infty}\frac{1}{2\pi i}\int_{\boldsymbol{\gamma}_{R,1}}\frac{1}{z-\lambda }R(\lambda:{\mathcal{A}})\,d\lambda=\frac{1}{2\pi i}\int_{\Gamma_{\theta,\varepsilon}}\frac{1}{z-\lambda}R(\lambda:{\mathcal{A}})\,d\lambda.
\end{equation*}
Therefore,~\eqref{eq:resolvent_formula} follows from~\eqref{eq:closed_con} by taking $R \to \infty$.

{ In order to prove \eqref{eq:product_Cauchy}, we use \eqref{eq:res_prop1} together with the   identity 
$$(z_1-\lambda)^{-1}(z_2-\lambda)^{-1} (z_2-z_1)=(z_1-\lambda)^{-1} - (z_2-\lambda)^{-1}$$ 
and apply the first part of the theorem. } 
\end{proof}

}

\subsection{Holomorphic calculi}

{Let ${\mathcal{A}}$ be a closed linear operator  satisfying \eqref{eq:res} and \eqref{eq:res_bound} Let $r>0$ be such that \eqref{eq:reslovent} holds. 
 For $f\in H_a^\infty$,  we put  }
\begin{equation}\label{eq:main_integral}
        f({\mathcal{A}}):=\frac{1}{2\pi i}\int_{\Gamma_{\theta,\varepsilon}}f(\lambda)R(\lambda:{\mathcal{A}})\,d\lambda, 
    \end{equation}
where $\Gamma_{\theta,\varepsilon}\subseteq \Sigma_{\delta}\cup B(0,r)$, $\pi/2<\theta<\pi/2+\delta$.  
Hence, by \eqref{bounded_line},  \eqref{eq:res_bound}, and \eqref{eq:reslovent}, the integral \eqref{eq:main_integral} converges absolutely {and defines a bounded operator on $X$}. 
The next proposition asserts that the integral does not depend on the path $\Gamma_{\theta,\varepsilon}$.

\begin{proposition}
    Suppose ${\mathcal{A}}$ satisfies \eqref{eq:res} and \eqref{eq:res_bound}. Let  and $\Gamma_{\theta_1,\varepsilon_1}, \Gamma_{\theta_2,\varepsilon_2}\subseteq \Sigma_\delta \cup B(0,r)$, $\varepsilon_1,\varepsilon_2\geq 0$.  Then 
    \begin{equation}
      \frac{1}{2\pi i}\int_{\Gamma_{\theta_1,\varepsilon_1}}f(\lambda)R(\lambda:{\mathcal{A}})\,d\lambda =    \frac{1}{2\pi i}\int_{\Gamma_{\theta_2,\varepsilon_2}}f(\lambda)R(\lambda:{\mathcal{A}})\,d\lambda.    
    \end{equation}
\end{proposition}
\begin{proof}
    By holomorphicity, for $\Gamma_{\theta,\varepsilon}, \,  \Gamma_{\theta,\varepsilon'} \subseteq \Sigma_\delta\cup B(0,r)$:
     \begin{equation}
      \frac{1}{2\pi i}\int_{\Gamma_{\theta,\varepsilon}}f(\lambda)R(\lambda:{\mathcal{A}})\,d\lambda =    \frac{1}{2\pi i}\int_{\Gamma_{\theta,\varepsilon'}}f(\lambda)R(\lambda:{\mathcal{A}})\,d\lambda.    
    \end{equation}
    Consider $\pi/2<\theta_1<\theta_2<\pi/2+\delta$. Let $0\leq \varepsilon_1<\varepsilon_2$ be such that  $\Gamma_{\theta_1,\varepsilon_1}, \Gamma_{\theta_2,\varepsilon_2}\subseteq \Sigma_\delta \cup B(0,r)$. Then,  using \eqref{eq:resolvent_formula}, we write 
    \begin{equation}\label{eq:two_int}
    \begin{split}
         \frac{1}{2\pi i}\int_{\Gamma_{\theta_1,\varepsilon_1}}f(\lambda)R(\lambda:{\mathcal{A}})\,d\lambda & = \frac{1}{2\pi i}\int_{\Gamma_{\theta_1,\varepsilon_1}}f(\lambda)\, 
         \frac{1}{2\pi i} \int_{\Gamma_{\theta_2,\varepsilon_2}} \frac{1}{\lambda-z} R(z:{\mathcal{A}})\, dz\, d\lambda.  
   \end{split} \end{equation}
   Since the double integral on the right-side of \eqref{eq:two_int} 
   is absolutely convergent, we apply the Fubini theorem together with Lemma \ref{lem:basic_lemma}  and get the proposition.  
\end{proof}
\begin{proposition}[cf. {\cite[Theorems 8.3 and 9.6]{Haase}}]\label{prop:calcul} 
    Suppose $f,g\in H_a^\infty$. Then  
    \begin{equation}\label{eq:homomor}
        (f\cdot g)({\mathcal{A}})=f({\mathcal{A}})g({\mathcal{A}}). 
    \end{equation}
\end{proposition}
\begin{proof}
{Clearly, if $f,g\in H^\infty_a$, then $f\cdot g\in H^\infty_a$.}     Fix $\pi/2<\theta_1<\theta_2<\theta<\pi $ and 
    $0\leq \varepsilon_1< \varepsilon_2<\varepsilon$, such that $\Gamma_{\theta,\varepsilon} \subseteq \Sigma_\delta\cup B(0,r)$. Then, using Lemma \ref{lem:basic_lemma}, we get  
    \begin{equation}
        \begin{split}
            (f\cdot g)({\mathcal{A}})&= \frac{1}{2\pi i} \int_{\Gamma_{\theta, \varepsilon} }
            f(\lambda)g(\lambda) R(\lambda:{\mathcal{A}})\, d\lambda\\
            & =  \frac{1}{2\pi i} \int_{\Gamma_{\theta, \varepsilon} }
           \frac{1}{2\pi i}  \int_{\Gamma_{\theta_1, \varepsilon_1} } \frac{f(z_1)}{z_1-\lambda} \, dz_1  \frac{1}{2\pi i}  \int_{\Gamma_{\theta_2, \varepsilon_2} } \frac{g(z_2)}{z_2-\lambda}\, dz_2\,  R(\lambda:{\mathcal{A}})\, d\lambda.\\
        \end{split}
    \end{equation}
Since the triple integral is absolutely convergent, we utilize the Fubini theorem together  with \eqref{eq:product_Cauchy} and obtain  \eqref{eq:homomor}. 
\end{proof}

{
\begin{remark}\normalfont From the facts that $\mathcal L(X)\ni B\mapsto B^*\in \mathcal L(X^*)$ is an isometric injection and the integral \eqref{eq:main_integral} converges absolutely, we deduce that 
\begin{equation}
    \label{main_itegral2} 
        f({\mathcal{A}})^*=\frac{1}{2\pi i}\int_{\Gamma_{\theta,\varepsilon}}f(\lambda)R(\lambda:{\mathcal{A}})^*\,d\lambda, 
\end{equation}  
{provided $\Gamma_{\theta,\varepsilon} \subseteq \Sigma_\delta\cup B(0,r)$ and $f\in H^\infty_a$.}
\end{remark}
}

\subsection{Subordinate semigroup and holomorphic calculi} Let us recall that for a uniformly bounded and strongly continuous semigroup $\{\mathcal{T}_t\}_{t \geq 0}$ its subordinate semigroup $\{\mathcal{P}_t\}_{t \geq 0}$ is defined by \eqref{eq:P_sub}. For the convince of the reader we provide a short proof that $\{\mathcal{P}_t\}_{t \geq 0}$ is obtained by the holomorphic calculi.

\begin{lemma}\label{lem:P_t_cauchy}
    Let $\{\mathcal{T}_t\}_{t \geq 0}$ be an analytic, uniformly bounded semigroup on a Banach  space $X$.  Let ${\mathcal{A}}$ denote its generator. Assume that $0 \in \rho({\mathcal{A}})$. There exists $\frac{\pi}{2}<\delta<\pi$ such that for all $\frac{\pi}{2}<\theta<\delta$ we have
    \begin{equation}\label{eq:P_t_holo}
        {\mathcal{P}}_t=\frac{1}{2\pi i}\int_{\Gamma_\theta}e^{-t\sqrt{-\lambda}}R(\lambda:{\mathcal{A}})\,d\lambda.
    \end{equation} 
\end{lemma}

\begin{proof}
   Let $\delta, r>0$ be such that \eqref{eq:res}, \eqref{eq:res_bound}, and \eqref{eq:reslovent} are satisfied. 
 Then ~\cite[Theorem 2.5.2]{Pazy}) asserts that 
    \begin{equation}\label{eq:T_path}
        \mathcal{T}_t=\frac{1}{2\pi i}\int_{\Gamma_\theta} e^{\lambda t}R(\lambda:{\mathcal{A}})\,d\lambda.
    \end{equation}
    Let us plug in~\eqref{eq:T_path} to~\eqref{eq:P_sub}. We obtain
    \begin{equation}\label{eq:double_integral}
        {\mathcal{P}}_t=\frac{1}{{\Gamma(1/2)}2\pi i}\int_{0}^{\infty}\int_{\Gamma_{\theta}} e^{-u}e^{\frac{t^2}{4u}\lambda}R(\lambda:{\mathcal{A}})\frac{1}{\sqrt{u}}\,d\lambda\,du.
    \end{equation}
    It is not difficult to check  that the double integral ~\eqref{eq:double_integral} is absolutely convergent. Hence, by the Fubini theorem,  
    we get
    \begin{align*}
        {\mathcal{P}}_t=\frac{1}{{\Gamma(1/2)}2\pi i}\int_{\Gamma_{\theta}}\int_{0}^{\infty} e^{-u}e^{-\frac{t^2}{4u}(-\lambda)}\frac{1}{\sqrt{u}}\,du\,  R(\lambda:{\mathcal{A}}) \,d\lambda=\frac{1}{2\pi i}\int_{\Gamma_\theta}e^{-t\sqrt{-\lambda}}R(\lambda:{\mathcal{A}})\,d\lambda. 
    \end{align*}
\end{proof}

\begin{corollary}\label{cor:cal-deriv}
    Under the assumptions of Lemma \ref{lem:P_t_cauchy}, for all non-negative  integers $n$ and $0\leq \varepsilon <\varepsilon_0$, $\varepsilon_0$ small enough, we have 
    \begin{equation}\label{eq:der_P_t_holo} \frac{d^n}{dt^n} {\mathcal{P}}_t=\frac{1}{2\pi i} \int_{\Gamma_{\theta,\varepsilon}} (-\sqrt{-\lambda})^n e^{ -t\sqrt{-\lambda}} R(\lambda:{\mathcal{A}})\, d\lambda.
    \end{equation}
\end{corollary}

\begin{remark}
    The formulae \eqref{eq:P_t_holo}, \eqref{eq:T_path}, and \eqref{eq:der_P_t_holo} hold for $\mathcal{T}_t^*$ and ${\mathcal{P}}_t^*$ by replacing $R(\lambda:{\mathcal{A}})$ by $R(\lambda:{\mathcal{A}})^*$. 
\end{remark}
\begin{proposition}\label{prop:crucial}
    Assume that ${\mathcal{A}}$ is an infinitesimal generator of a uniformly bounded holomorphic semigroup {$\{\mathcal{T}_t\}_{t \geq 0}$} in a sector satisfying \eqref{eq:res} and \eqref{eq:res_bound}. Fix a positive integer $n$. Then there is a constant $C_n>0$, which depends only on $n$, $\delta$,  and the constant $C$ in \eqref{eq:res_bound}, such that for all $X^*\in X^*$,  if $\eta(t)=\mathcal{T}_t^*x^*$ and $\zeta (t)={\mathcal{P}}_t^*x^*$, then 
    \begin{equation}
        \| \eta^{(n)}(t)\|_{X^*} \leq C_n t^{-n/2} \| \zeta^{(n)}(\sqrt{t})\|_{X^*}. 
    \end{equation}
\end{proposition}
\begin{proof} Using \eqref{eq:T_path},  the holomorphic calculus {(Proposition \ref{prop:calcul}), and Corollary \ref{cor:cal-deriv},} we have 
    \begin{equation}
        \begin{split}
            \eta^{(n)}(t) & =\frac{1}{2\pi i} \int_{\Gamma_\theta} \lambda^n e^{t\lambda}R(\lambda:{\mathcal{A}})^* x^*\, d\lambda\\
            &= (-1)^n\frac{1}{2\pi i} \int_{\Gamma_\theta} (-\sqrt{-\lambda})^n e^{t\lambda +\sqrt{t}{\sqrt{-\lambda}}}  (-\sqrt{-\lambda})^n e^{-t\sqrt{-\lambda}} R(\lambda:{\mathcal{A}})^* x^*\, d\lambda\\
            &=  t^{-n/2}  m(t{\mathcal{A}})^* \zeta^{(n)}(t), 
        \end{split}
    \end{equation}
    where $m(\lambda) = (\sqrt{-\lambda})^n e^{\lambda +\sqrt{-\lambda}}$. Observe that $m(t\cdot)\in H^\infty_a$ and $\| m(t{\mathcal{A}})\|\leq C_n$ with $C_n$   which depends on $n\geq 1$ and the constant $C$ in \eqref{eq:res_bound}, but independent of $t>0$. 
\end{proof}

\subsection{Proof of the inclusion $\Lambda_{-\sqrt{-{\mathcal{A}}}}^{2\beta}\subseteq \Lambda_{\mathcal{A}}^{\beta}$ } 

In this subsection we complete the proof of Theorem  \ref{theo:Lamba=Lambda}. We assume that ${\mathcal{A}}$ is an infinitesimal generator of a uniformly bounded holomorphic $c_0$-semigroup $\{\mathcal{T}_t\}_{t \geq 0}$ on a Banach space $X$. 
Thus $\lambda I-{\mathcal{A}}$ is an invertible operator for $\lambda \in \Sigma_\delta \setminus \{0\}$ 
and 
$$ \| R(\lambda:{\mathcal{A}})\|\leq \frac{C}{|\lambda|} \quad \text{\rm for } \lambda \in \Sigma_\delta \setminus \{0\}.$$
Let us remark that we do not assume that $0\in\rho ({\mathcal{A}})$.

Consider the approximation semigroups $\mathcal{T}_t^{\omega}=e^{-\omega t}\mathcal{T}_t$, $0<\omega<1$,  
generated by ${\mathcal{A}}_\omega= {\mathcal{A}}-\omega I $ and the associated Poisson semigroups 
$$ {\mathcal{P}}_t^{\omega} ={\Gamma (1/2)^{-1}} \int_0^\infty e^{-u} \mathcal{T}_{t^2/4u}^\omega \frac{du}{\sqrt{u}}.$$
Clearly, $0\in\rho ({\mathcal{A}}_\omega)$, $R(\lambda : {\mathcal{A}}_\omega)=R(\omega +\lambda: {\mathcal{A}})$ and assuming that $0<\delta<\pi/2$, we have 
\begin{equation}
    \|R(\lambda:{\mathcal{A}}_\omega)\|\leq \frac{C'}{|\lambda|} \quad \text{for } \lambda \in \Sigma_\delta\setminus \{0\}
\end{equation}
with $C'$ independent of $\omega \in (0,1)$. 

It is not difficult to prove that for each non-negative integer $n$, 
\begin{equation}\label{eq:limits}
    \lim_{\omega \to 0} \frac{d^n}{dt^n} \mathcal{T}^{\omega}_t = \frac{d^n}{dt^n}  \mathcal{T}_t, \quad \lim_{\omega \to 0}\frac{d^n}{dt^n} {\mathcal{P}}^{\omega}_t = \frac{d^n}{dt^n} {\mathcal{P}}_t 
\end{equation}
in the norm operator topology uniformly for $t$ being in any compact set  in $(0,\infty)$. 
{The same convergences hold in the norm operator topology of $X^*$ } for  $\mathcal{T}_t^{\omega *}$, $\mathcal{T}_t^*$, ${\mathcal{P}}_t^{\omega *}$, and ${\mathcal{P}}_t^*$. 

\begin{proof}[Proof of the inclusion  $\Lambda_{-\sqrt{-{\mathcal{A}}}}^{2\beta}\subseteq \Lambda_{\mathcal{A}}^{\beta}$ and the first inequality in \eqref{eq:Lambda=Lambda}] Suppose  $x^*\in \Lambda^{2\beta}_{-\sqrt{-{\mathcal{A}}}}$.  Fix a positive integer $n>2\beta$. Then, for all $t>0$,  
\begin{equation}\label{eq:xLambda}
    \Big\| \frac{d^n}{dt^n} {\mathcal{P}}_t^*x^*\Big\|_{X^*}\leq  t^{2\beta -n}\| x^*\|_{\Lambda_{-\sqrt{-{\mathcal{A}}}}^{2\beta}}.
\end{equation}
Set  $\eta_\omega (t)=\mathcal{T}_t^{\omega *} x^*$, ${\zeta_\omega (t)={\mathcal{P}}_t^{\omega *} x^*}$. 
Then 
\begin{equation}\label{eq:limits2}
    \lim_{\omega\to 0} \eta^{(n)}_\omega (t)=\frac{d^n}{dt^n} \mathcal{T}_t^*x^*, 
\end{equation}
\begin{equation}\label{eq:limits3}
    \lim_{\omega\to 0} \zeta^{(n)}_\omega (t)=\frac{d^n}{dt^n} {\mathcal{P}}_t^*x^*, 
\end{equation}
where the convergences are in the norm in $X^*$ and uniform in any interval $[a,b]\subset (0,\infty)$. 
Applying Proposition 
\ref{prop:crucial}, we have 
\begin{equation}\label{eq:eta_zeta}
    \|\eta_\omega^{(n)} \|_{X^*}\leq C_n{\|\zeta_\omega^{(n)} (\sqrt{t})\|_{X^*}}, \quad \omega \in (0,1), \ t>0.
\end{equation} 
Now, from \eqref{eq:limits2}    and \eqref{eq:eta_zeta}, we conclude that 
\begin{equation*}\begin{split}
    \Big\| \frac{d^n}{dt^n} \mathcal{T}_t^*x^*\Big\|_{X^*}&=\lim_{\omega\to 0} \Big\| \eta_\omega^{(n)} (t)\Big\|_{X^*}\leq \lim_{\omega\to 0} CC_n t^{-n/2} \| \zeta_\omega^{(n)} (\sqrt{t})\|_{X^*}\\
    &\leq CC_n t^{-n/2} \sqrt{t}^{2\beta-n} \| x^*\|_{\Lambda_{-\sqrt{-{\mathcal{A}}}}^{2\beta}}=CC_n t^{\beta-n}\| x^*\|_{\Lambda_{-\sqrt{-{\mathcal{A}}}}^{2\beta}}
\end{split} \end{equation*}
where in the last inequality we have used \eqref{eq:limits3} and \eqref{eq:xLambda}.  The required inclusion and the first inequality in \eqref{eq:Lambda=Lambda} is established. 
\end{proof}
\section{ $\Lambda^\beta_{\mathcal A}$-spaces associated with special forms of operators}\label{sec:A-special} 

Let $\{\mathcal{T}_t\}_{t \geq 0}$ be a uniformly bounded strongly continuous analytic semigroup on a Banach space $X$. Let us denote ${\frak D} =\{\mathcal{T}_tx: x\in X, \ t>0\}$   the space of test vectors,  which, by the strong continuity,  is dense in $X$.  We assume that its infinitesimal generator ${\mathcal{A}}$ has a special form, namely 
\begin{equation}\label{eq:A-form} {\mathcal{A}}\mathcal{T}_tx=\sum_{j=1}^n \epsilon_j \mathcal D_j\mathcal D_j\mathcal{T}_tx, \quad t>0, \ x\in X,
\end{equation}
where 
$\epsilon_j\in \mathbb C$, {$n \in \mathbb{N}$,} and {linear operators} $\mathcal D_j:\frak D\to \frak D$ satisfy 
\begin{equation}\label{eq:commute} \mathcal D_j\mathcal{T}_{t+s}x=\mathcal{T}_t\mathcal D_j\mathcal{T}_sx, \quad x\in X;
\end{equation}
\begin{equation}\label{eq:D_j_bounds} \| \mathcal D_j\mathcal{T}_tx\|\leq C_j t^{-1/2} \|x\|, \quad x\in X.
\end{equation}
Let us note that the conditions imply that the mappings $(0,\infty)\ni t\mapsto \mathcal D_j\mathcal{T}_t\in \mathcal L(X)$ are $C^\infty$, and 
$$ \frac{d}{dt} \mathcal D_j\mathcal{T}_t=\mathcal D_j\mathcal{T}_{t_1} {\mathcal{A}}\mathcal{T}_{t_2}= {\mathcal{A}}\mathcal{T}_{t_2}=\mathcal D_j\mathcal{T}_{t_1}, \quad t_1+t_2=t. $$
The same conclusions hold for $(\mathcal D_j\mathcal{T}_t)^*$, namely  for all $1 \leq j \leq n$ and all $t,s>0$, we have 
\begin{equation}
    (\mathcal D_j\mathcal{T}_{t+s})^* = \mathcal{T}_t^*(\mathcal D_j\mathcal{T}_s)^*, 
\end{equation}
\begin{equation}\label{eq:djbound}
    \| (\mathcal D_j\mathcal{T}_t)^*\|_{X^*}\leq C_j t^{-1/2},
\end{equation}
\begin{equation}
    \frac{d}{dt} (\mathcal D_j\mathcal{T}_t)^*= (\mathcal D_j\mathcal{T}_{t_1})^*({\mathcal{A}}\mathcal{T}_{t_2} )^*= ({\mathcal{A}}\mathcal{T}_{t_2} )^*(\mathcal D_j\mathcal{T}_{t_1})^*, \quad t_1+t_2=t.  
\end{equation}
For $x^*\in X^*$, we say that $\mathcal D_j^* x^*$ belongs to $X^*$ in the mild sense,  if there is $y_j^*\in X^*$  such that 
\begin{equation}\label{eq:weak-sense} \langle y_j^*,\mathcal{T}_tx\rangle = \langle x^* , \mathcal D_j\mathcal{T}_tx\rangle 
\end{equation}
for all $x\in X$ and $t>0$. Then we write $y^*_j=\mathcal D_j^*x^*$. 

{ \begin{example}
    On $\mathbb{R}^2$, let us consider the operator $\mathcal{A}=-(\partial_1)^4+(\partial_2)^2$. It is a generator of an analytic semigroup $\{\mathcal{T}_t\}_{t \geq 0}$ on $L^2(d\mathbf{x})$ such that  $\mathcal{T}_tf=f*k_t$, where $\hat{k_t}(\xi)=e^{-t\xi_1^4+t\xi_2^2}$, $\xi=(\xi_1,\xi_2)$. One can prove that the operator $\mathcal{A}$ and the semigroup $\{\mathcal{T}_t\}_{t \geq 0}$ fit the scope described above. Moreover, for $0<\beta<1$, one can check that $f \in \Lambda_{{\mathcal{A}}}^{\beta}$ if and only if there is a constant $C>0$ such that for all $\mathbf{x},h \in \mathbb{R}^2$, $h=(h_1,h_2)$ we have 
    \begin{align*}
        |f(\mathbf{x})-f(\mathbf{x}+h)| \leq C(|h_1|^{\beta}+|h_2|^{\beta/2}).
    \end{align*}
\end{example}}

Our aim is to prove the following theorem. 

\begin{theorem}\label{teo:A-special}
    Assume that $\beta>1/2$. Then $x^*\in\Lambda_{\mathcal{A}}^{\beta}$ if and only if $\mathcal D^*_jx^* \in \Lambda_{\mathcal{A}}^{\beta-1/2}$ for all $1 \leq j \leq n$, where the action of $\mathcal D^*_j$ on $x^*$ is understood in the mild sense. Moreover,  there is a constant $C>1$ such that for all $x^* \in X^*$, we have 
    $$ C^{-1}\|x^*\|_{\Lambda_{\mathcal{A}}^{\beta}} \leq \|x^*\|_{X^*}+\sum_{j=1}^n \| \mathcal D_j^*x^*\|_{\Lambda_{\mathcal{A}}^{\beta -1/2}} \leq C  \|x^*\|_{\Lambda_{\mathcal{A}}^{\beta}} .$$
\end{theorem}
\begin{proof}
    The proof mimics that of Theorem \ref{theo:A_Lambda}. 
    Assume that $x^*\in \Lambda_{\mathcal{A}}^\beta$. Fix a positive integer $m>\beta$ and $1 \leq j \leq n$. Consider the $C^\infty$-function 
    $$(0,\infty)\ni t\mapsto v_j(t)=(\mathcal D_j \mathcal{T}_t)^*x^*\in X^*.$$
    Then 
    \begin{equation}\label{eq:der_v_j} \|v_j^{(m)}(t)\|_{X^*}\leq C\| x^*\|_{\Lambda_{\mathcal{A}}^\beta} t^{\beta-m-1/2}.\end{equation}
    We write the Taylor expansion of $v_j$ at $t_0=1$:
    $$ v_j(t)=\sum_{\ell=0}^{m-1} \frac{1}{\ell !} v_j^{\ell}(1)(t-1)^\ell +\int_1^t\frac{(t-s)^{m-1}}{(m-1)!} v_j^{(m)}(s)\, ds.$$
    It follows from \eqref{eq:der_v_j} that $v_j(t)$ converges in the $X^*$-norm,  as $t$ tends to 0,  to a vector denoted by  $y_j^*$.  To prove that $y_j^*\in \Lambda_{\mathcal{A}}^{\beta -1/2}$, we write 
    \begin{equation*}
    \begin{split}
        \Big\| \frac{d^m}{dt^m} \mathcal{T}_t^*y_j^*\Big\|_{X^*} & = \sup_{\|x\|=1} \Big|\Big\langle y_j^*, \frac{d^m}{dt^m} \mathcal{T}_t x\Big\rangle \Big| 
        =\sup_{\|x\|=1} \lim_{s\to 0} \Big|\Big\langle (\mathcal D_j\mathcal{T}_s)^* x^*,  \frac{d^m}{dt^m} \mathcal{T}_t x\Big\rangle \Big|\\
        &=\sup_{\|x\|=1} \lim_{s\to 0} |\langle x^* , {\mathcal{A}}^m \mathcal{T}_{t/2} \mathcal D_j\mathcal{T}_{t/2} \mathcal{T}_sx\rangle|\\
        &\leq \sup_{\|x \|=1} \lim\sup_{s\to 0} \| ({\mathcal{A}}^m\mathcal{T}_{t/2})^* x^*\|_{X^*} \| \mathcal D_j\mathcal{T}_{t/2} \mathcal{T}_sx\|\leq \|x^*\|_{\Lambda_{\mathcal{A}}^\beta} \Big(\frac{t}{2}\Big)^{\beta -m}\Big( \frac{t}{2} \Big)^{-1/2},  
    \end{split} \end{equation*}
where in the last inequality we have used~\eqref{eq:djbound}. Hence, the claim is proved since $m>\beta-1/2$. 

We now prove the converse implication. Suppose $y_j^*=\mathcal D^*_j x^*\in \Lambda_{\mathcal{A}}^{\beta-1/2}$, $j=1,2,...,n$. Fix $m>\beta +1/2$. Then 
\begin{equation*}
\begin{split}
    &\Big\|\frac{d^m}{dt^m} \mathcal{T}_t^* x^*\Big\|_{X^*} =\sup_{\|x\|=1} |\langle x^*, {\mathcal{A}}^m\mathcal{T}_tx\rangle| 
    = \sup_{\|x\|=1} |\langle x^*, \sum_{j=1}^n \epsilon_j \mathcal D_j \mathcal{T}_{t/2} {\mathcal{A}}^{m-1}  \mathcal D_j\mathcal{T}_{t/2}x\rangle|\\
    &\leq \sum_{j=1}^n \sup_{\|x\|=1} |\langle y_j^*, {\mathcal{A}}^{m-1} \mathcal{T}_{t/2} \mathcal D_j \mathcal{T}_{t/2}x\rangle| 
    \leq \sum_{j=1}^n \sup_{\|x\|=1} |\langle ({\mathcal{A}}^{m-1}\mathcal{T}_{t/2})^* y_j^*,  \mathcal D_j\mathcal{T}_{t/2}x\rangle| \\
    &\leq \sum_{j=1}^n \|y^*_j\|_{\Lambda_{\mathcal{A}}^{\beta-1}} 2^{m} t^{\beta -1/2 -(m-1)} C_j t^{-1/2}= \Big(2^m\sum_{j=1}^n C_j  \|y^*_j\|_{\Lambda_{\mathcal{A}}^{\beta-1}}\Big)t^{\beta -m}.
\end{split}\end{equation*}
\end{proof}
\section{Interpolation}\label{sec:interpolation}
{  For $0<\beta_0<\beta_1$, let us consider the interpolation couple $\{\Lambda_{\mathcal A}^{\beta_0}, \Lambda_{\mathcal A}^{\beta_1}\}$. Clearly, {by~\eqref{eq:inclusion_trivial},} $\Lambda_{\mathcal A}^{\beta_0} + \Lambda_{\mathcal A}^{\beta_1}=\Lambda_{\mathcal A}^{\beta_0}$. For $0<t<\infty$, set 
\begin{equation}
    K(t,x^*):=\inf_{x^*=x_0^*+x_1^*} (\|x_0^*\|_{\Lambda_{\mathcal A}^{\beta_0}} +t\|x_1^*\|_{\Lambda_{\mathcal A}^{\beta_1}}), \quad x_0^*\in \Lambda_{\mathcal A}^{\beta_0}, \ x_1^*\in \Lambda_{\mathcal A}^{\beta_1}. 
\end{equation}
Since $  \Lambda_{\mathcal A}^{\beta_1} \subseteq \Lambda_{\mathcal A}^{\beta_0}$, we conclude that {there is a constant $c \in (0,1]$ such that 
\begin{equation}\label{eq:t>1} 
 c\|x^*\|_{\Lambda^{\beta_0}_{\mathcal A} }\leq    K(t,x^*)\leq \| x^*\|_{\Lambda_{\mathcal A}^{\beta_0}} \quad \text{ for } t\geq 1 \ \text{ and }  x^*\in \Lambda_{\mathcal A}^{\beta_0}. 
\end{equation}  Indeed, there is $0<c\leq 1$ such that  if $x^*=x_0^*+x_1^*$, $x_0^*\in \Lambda_{\mathcal A}^{\beta_0}$, $x_1^*\in \Lambda_{\mathcal A}^{\beta_1}$, then 
\begin{equation*}\begin{split}
    \| x_0^*\|_{\Lambda_{\mathcal A}^{\beta_0}} + t  \| x_1^*\|_{\Lambda_{\mathcal A}^{\beta_1}} &\geq  \| x_0^*\|_{\Lambda_{\mathcal A}^{\beta_0}} + c \| x_1^*\|_{\Lambda_{\mathcal A}^{\beta_0}} \geq c \| x_0^*+x_1^*\|_{\Lambda^{\beta_0}_{\mathcal A}}=c \| x^*\|_{\Lambda^{\beta_0}_{\mathcal A}}, 
\end{split} \end{equation*}
which proves $K(t,x^*)\geq c \|x^*\|_{\Lambda^{\beta_0}_{\mathcal A}}$ for $t\geq 1$. } The second inequality in \eqref{eq:t>1} is obvious, by taking $x_0^*=x^*$ and $x_1^*=0$. 

For $0<\theta<1$, the interpolation intermediate space defined by the $K$-method of Peetre is given by 
\begin{equation}
    (\Lambda_{\mathcal A}^{\beta_0}, \Lambda_{\mathcal A}^{\beta_1})_\theta =(\Lambda_{\mathcal A}^{\beta_0}, \Lambda_{\mathcal A}^{\beta_1})_{\theta, \infty} =\{x^*\in \Lambda_{\mathcal A}^{\beta_0}:\sup_{t>0}t^{-\theta} K(t,x^*)=: \|x^*\|_\theta <\infty\},
\end{equation}
see e.g. \cite[Chapter III]{Butzer}, \cite[Section 1.3.3]{Triebel}. 

\begin{proof}[Proof of Theorem \ref{teo:interpolation}]
    We start by proving the relations 
    \begin{equation}\label{eq:first_inclusion}
        (\Lambda_{\mathcal A}^{\beta_0}, \Lambda_{\mathcal A}^{\beta_1})_\theta\subseteq \Lambda_{\mathcal A}^\beta, \quad \| x^*\|_{\Lambda_{\mathcal A}^\beta}\leq C \| x^*\|_\theta.  
    \end{equation}
    Let $x^* \in  (\Lambda_{\mathcal A}^{\beta_0}, \Lambda_{\mathcal A}^{\beta_1})_\theta$. Then 
    \begin{equation}\label{eq:inter1}
        \| x^*\|_{\theta} \geq K(1,x^*)\geq C\| x^*\|_{\Lambda_{\mathcal A}^{\beta_0}}\geq C'  \| x^*\|_{X^*}. 
    \end{equation}
Fix a positive integer $m>\beta_1$. It suffices to prove that there is $C>0$ such that for all $t>0$, we have 
\begin{equation}
    \| (\mathcal A\mathcal T_{t/m})^{*m} x^*\|_{X^*}\leq C t^{\beta -m} \| x^*\|_\theta. 
\end{equation}
To this end, we consider $K(t^{\beta_1-\beta_0}, x^*)$. By the definition of $K(t^{\beta_1-\beta_0}, x^*)$, there are $x_0^*\in \Lambda_{\mathcal A}^{\beta_0}$ and $x_1^*\in \Lambda_{\mathcal A}^{\beta_1}$ such that $x^*=x_0^*+x_1^*$ and 
\begin{equation}\label{eq:interp}
    t^{-(\beta_1-\beta_0)\theta} (\| x_0^*\|_{\Lambda_{\mathcal A} ^{\beta_0}} + t^{\beta_1-\beta_0} \| x_1^*\|_{\Lambda_{\mathcal A}^{\beta_1} })  \leq 2\| x^*\|_\theta. 
\end{equation}
Consequently, 
\begin{equation}\label{eq:inter2}
    \begin{split}
        \| (\mathcal A\mathcal T_{t/m})^{*m} x^*\|_{X^*} &\leq   \| (\mathcal A\mathcal T_{t/m})^{*m} x_0^*\|_{X^*}     +
        \| (\mathcal A\mathcal T_{t/m})^{*m} x_1^*\|_{X^*}  \\
        &\leq t^{\beta_0-m}\| x_0^*\|_{\Lambda_{\mathcal A}^{\beta_0}} + t^{\beta_1-m}\| x_1^*\|_{\Lambda_{\mathcal A}^{\beta_1}} \\
        &= t^{\beta -m} ( t^{-(\beta_1-\beta_0)\theta }\|x_0^*\|_{\Lambda_{\mathcal A}^{\beta_0} } + t^{(\beta_1-\beta_0)(1-\theta)} \|x_1^*\|_{\Lambda_{\mathcal A}^{\beta_1} })\\
        &\leq 2t^{\beta-m} \| x^*\|_\theta,
        \end{split}
\end{equation}
where in the last inequality we have used \eqref{eq:interp}. Thus \eqref{eq:first_inclusion} follows from \eqref{eq:inter1} and \eqref{eq:inter2}.

We now turn to prove the {inverse relations  to \eqref{eq:first_inclusion}.} Let $x^*\in \Lambda_{\mathcal A}^\beta$. Our first goal, for any fixed $0<t<1$, is to decompose 
\begin{equation}\label{eq:000}
    x^*=x_0^*+x_1^*, \quad t^{-\theta} 
    (\|x_0^{*}\|_{\Lambda_{\mathcal A}^{\beta_0} }
    +t\|x_1^*\|_{\Lambda_{\mathcal A}^{\beta_1}})\leq C\|   
 x^*\|_{\Lambda_{\mathcal A}^{\beta}} 
\end{equation}
with a constant $C>0$ independent of $x^* \in X^*$ and $0<t<1$. Set $v(s):=\mathcal T_s^*x^*$. Let $m$ be the smallest integer satisfying  $m\geq \beta$. Let $\tau=t^{1/(\beta_1-\beta_0)}$. Using the Taylor expansion of $v(s)$ at $\tau$, we get {

\begin{equation}\label{eq:x0x1}
     x^*=v(0)=\Big\{\sum_{\ell=0}^{m-1} \frac{1}{\ell !} v^{(\ell)} (\tau)(-\tau)^{\ell} \Big\}+ \Big\{\int_\tau^0 \frac{(-s)^{m-1}}{(m-1)!} v^{(m)}(s)\, ds\Big\}=:x_1^*+x_0^*.
\end{equation}
 }
We shall prove that 
\begin{equation}\label{eq:00}
     \| x_1^*\|_{\Lambda_{\mathcal A}^{\beta_1}}\leq C t^{-(1-\theta)}\|x^*\|_{\Lambda^{\beta}_{\mathcal A}} \quad  \text{\rm for } 0<t<1.
\end{equation}
For this purpose, it suffices to verify that  for $0\leq \ell\leq m-1$, one has 
\begin{equation}\label{eq:0}
    \| v^{(\ell)} (\tau)\tau^{\ell}\|_{\Lambda_{\mathcal A}^{\beta_1}}\leq C t^{-(1-\theta)}\|x^*\|_{\Lambda^{\beta}_{\mathcal A}} \quad  \text{\rm for } 0<t<1.
\end{equation}
First observe that, {by Lemma \ref{lem:converge}, }  
\begin{equation}\label{eq:1}
    \|  v^{(\ell)}(\tau)\|_{X^*}\leq C \|x^*\|_{\Lambda^{\beta}_{\mathcal A}}, \quad 0\leq \ell \leq m-1. 
\end{equation}
Fix an integer $n>\beta_1$. Then  by Proposition~\ref{pro:any_m}, for $0<s<1$, 
\begin{equation}\label{eq:2}
\begin{split}
\Big\| \frac{d^n}{ds^n} \mathcal T_s^* (v^{(\ell)}(\tau) )\tau^\ell \Big\|_{X^*} & \leq \| v^{(\ell +n)} (\tau +s)\tau^\ell \|_{X^*} \\
&\leq (\tau +s)^{\beta -\ell -n}\tau^{\ell} \| x^*\|_{\Lambda_{\mathcal A}^{\beta}} \\
&= (\tau+s)^{-(1-\theta)(\beta_1-\beta_0) + (\beta_1-n)-\ell}\tau^\ell  \| x^*\|_{\Lambda_{\mathcal A}^{\beta}}\\
&\leq s^{\beta_1-n} (\tau+s)^{-(1-\theta)(\beta_1-\beta_0)-\ell}\tau^\ell  \| x^*\|_{\Lambda_{\mathcal A}^{\beta}} \\
& \leq s^{\beta_1-n} \tau^{-(1-\theta)(\beta_1-\beta_0)}  \| x^*\|_{\Lambda_{\mathcal A}^{\beta}} \\
& = s^{\beta_1-n}  t^{-(1-\theta)} \| x^*\|_{\Lambda_{\mathcal A}^{\beta}}. \\
\end{split} \end{equation}
Now \eqref{eq:0} follows from \eqref{eq:1} and \eqref{eq:2}.  So,   \eqref{eq:00} is established. 

We now turn to examine $x_0^*$  defined in ~\eqref{eq:x0x1}. Fix $0\leq \varepsilon\leq \beta_0$ such that  $\beta-\varepsilon\notin \mathbb Z$ and 
$0<m-(\beta-\varepsilon) <1$ 
(if $\beta\notin \mathbb Z$, then we take $\varepsilon =0$). Then   $x^*\in \Lambda^{\beta-\varepsilon}_{\mathcal A}$. Consequently,  
\begin{equation}\label{eq:3}
    \begin{split}
        \|x_0^*\|_{X^*}&\leq C\Big(\int_0^\tau s^{m-1} s^{\beta -\varepsilon -m}\, ds \Big) \| x^*\|_{\Lambda^{\beta-\varepsilon}_{\mathcal A}} \leq C \tau^{\beta -\varepsilon }  \| x^*\|_{\Lambda^{\beta}_{\mathcal A}} \\
        &=C t^{(\beta-\varepsilon)/(\beta_1-\beta_0)} \| x^*\|_{\Lambda^\beta_{\mathcal A}} 
        \leq Ct^{\theta}  \| x^*\|_{\Lambda^\beta_{\mathcal A}},
    \end{split}
\end{equation}
provided $\varepsilon$ is small enough. 
Further, let $n$ be a fixed positive integer such that $n>\beta_0$. Then, by Proposition~\ref{pro:any_m}, 
\begin{equation}\label{eq:4} 
    \begin{split}
        \Big\| \frac{d^n}{d u^n} \mathcal T^*_u x_0^*\Big\|_{X^*} &\leq C \int_0^\tau s^{m-1} \|v^{(m+n)}(s+u)\|\, ds\\
        &\leq C\Big( \int_0^\tau s^{m-1} (s+u)^{\beta -m-n}\, ds\Big)\| x^*\|_{\Lambda_{\mathcal A}^{\beta}}\\
        &= C\Big(\int_0^\tau s^{m-1} (s+u)^{\theta(\beta_1-\beta_0)-m} (s+u)^{\beta_0-n}\, ds \Big)\| x^*\|_{\Lambda_{\mathcal A}^{\beta}}\\
        &\leq C' u^{\beta_0-n} \Big(\int_0^\tau s^{m-1} s^{\theta(\beta_1-\beta_0)-m}\, ds\Big) \| x^*\|_{\Lambda_{\mathcal A}^{\beta}}\\
        &\leq C'  u^{\beta_0-n} \tau^{\theta(\beta_1-{\beta_0})} \| x^*\|_{\Lambda_{\mathcal A}^{\beta}}= C  u^{\beta_0-n}t^\theta  \| x^*\|_{\Lambda_{\mathcal A}^{\beta}}. 
    \end{split}
\end{equation}
From \eqref{eq:3}, \eqref{eq:4}, and Proposition~\ref{pro:any_m}, we conclude that 
\begin{equation}\label{eq:5}
\| x_0^*\|_{\Lambda_{\mathcal A}^{\beta_0}}\leq Ct^\theta  \| x^*\|_{\Lambda_{\mathcal A}^{\beta}}, \quad 0<t<1.
\end{equation}
Combining together \eqref{eq:5} and \eqref{eq:00}, we obtain \eqref{eq:000}. 
Consequently,  we have proved that 
$$\sup_{0<t<1} t^{-\theta} K(t,x^*)\leq C\| x^*\|_{\Lambda^\beta_{\mathcal A}}.$$
{If $t\geq 1$, then using \eqref{eq:t>1}, we 
get $ \sup_{t\geq 1} t^{-\theta}K(t,x^*)\leq \| x^*\|_{\Lambda^{\beta_0}_{\mathcal A} } \leq C\| x^*\|_{\Lambda^{\beta}_{\mathcal A} }$. Consequently, 
$\| x^*\|_{\theta}=\sup_{t>0} t^{-\theta} K(t,x^*)\leq C \| x^*\|_{\Lambda^\beta_{\mathcal A}}$.}  
 \end{proof}
}

\part{Lipschitz spaces in the Dunkl setting}\label{part2}

\section{Preliminaries}

\subsection{Dunkl theory}

In this section we present basic facts concerning the theory of the Dunkl operators.   For more details we refer the reader to~\cite{Dunkl},~\cite{Roesle99},~\cite{Roesler3}, and~\cite{Roesler-Voit}. 

We consider the Euclidean space $\mathbb R^N$ with the scalar product $\langle \mathbf{x},\mathbf y\rangle=\sum_{j=1}^N x_jy_j
$, where $\mathbf x=(x_1,...,x_N)$, $\mathbf y=(y_1,...,y_N)$, and the norm $\| \mathbf x\|^2=\langle \mathbf x,\mathbf x\rangle$.

A {\it normalized root system}  in $\mathbb R^N$ is a finite set  $R\subset \mathbb R^N\setminus\{0\}$ such that $R \cap \alpha  \mathbb{R} = \{\pm \alpha\}$,  $\sigma_\alpha (R)=R$, and $\|\alpha\|=\sqrt{2}$ for all $\alpha\in R$, where $\sigma_\alpha$ are defined by 
\begin{equation}\label{reflection}\sigma_\alpha (\mathbf x)=\mathbf x-2\frac{\langle \mathbf x,\alpha\rangle}{\|\alpha\|^2} \alpha.
\end{equation}

The finite group $G$ generated by the reflections $\sigma_{\alpha}$, $\alpha \in R$, is called the {\it Coxeter group} ({\it reflection group}) of the root system.

A~{\textit{multiplicity function}} is a $G$-invariant function $k:R\to\mathbb C$ which will be fixed and $\geq 0$  throughout this paper.  

The associated measure $dw$ is defined by $dw(\mathbf x)=w(\mathbf x)\, d\mathbf x$, where 
 \begin{equation}\label{eq:measure}
w(\mathbf x)=\prod_{\alpha\in R}|\langle \mathbf x,\alpha\rangle|^{k(\alpha)}.
\end{equation}
Let $\mathbf{N}=N+\sum_{\alpha \in R}k(\alpha)$. Then, 
\begin{equation}\label{eq:t_ball} w(B(t\mathbf x, tr))=t^{\mathbf{N}}w(B(\mathbf x,r)) \ \ \text{\rm for all } \mathbf x\in\mathbb R^N, \ t,r>0.
\end{equation}
Observe that there is a constant $C>1$ such that for all $\mathbf{x} \in \mathbb{R}^N$ and $r>0$, we have
\begin{equation}\label{eq:balls_asymp}
C^{-1}w(B(\mathbf x,r))\leq  r^{N}\prod_{\alpha \in R} (|\langle \mathbf x,\alpha\rangle |+r)^{k(\alpha)}\leq C w(B(\mathbf x,r)),
\end{equation}
so $dw(\mathbf x)$ is doubling.

For $\xi \in \mathbb{R}^N$, the {\it Dunkl operators} $D_\xi$  are the following $k$-deformations of the directional derivatives $\partial_\xi$ by   difference operators:
\begin{equation}\label{eq:T_xi}
     D_\xi f(\mathbf x)= \partial_\xi f(\mathbf x) + \sum_{\alpha\in R} \frac{k(\alpha)}{2}\langle\alpha ,\xi\rangle\frac{f(\mathbf x)-f(\sigma_\alpha(\mathbf{x}))}{\langle \alpha,\mathbf x\rangle}.
\end{equation}
 We simply write $D_j$, if $\xi=e_j$, {where $\{e_j\}_{j=1}^N$ stands for the canonical basis in $\mathbb R^N$.}  

The Dunkl operators $D_{\xi}$, which were introduced in~\cite{Dunkl}, commute and are skew-symmetric with respect to the $G$-invariant measure $dw$. For multi-index $\boldsymbol{\beta}=(\beta_1,\beta_2,\ldots,\beta_N)\in \mathbb N_0^N$, we define
\begin{equation}\label{eq:iterated_der_ord}
    |\boldsymbol{\beta}|=\beta_1+\ldots +\beta_N, \ \partial^{\mathbf{0}}=I, \ \ \partial^{\boldsymbol{\beta}}=\partial_1^{\beta_1} \circ \ldots \circ \partial_N^{\beta_N}, \ \ D^{\mathbf{0}}=I, \ \ D^{\boldsymbol{\beta}}=D_{1}^{\beta_1} \circ \ldots \circ D_{N}^{\beta_N}
\end{equation}

 Let $f$ be a bounded measurable function. Fix a multi-index $\boldsymbol\beta$. We say that $D^{\boldsymbol\beta}f$ belongs to $L^\infty$ in the sense of distribution { $\mathcal S'_{\rm Dunkl}(\mathbb R^N)$}, if there is a bounded function $g$ such that 
\begin{equation}
    \label{def:distr} \int_{\mathbb{R}^N} g(\mathbf x)\varphi(\mathbf x)\, dw(\mathbf x)=(-1)^{|\boldsymbol\beta|}\int_{\mathbb{R}^N} f(\mathbf x)D^{\boldsymbol \beta} \varphi(\mathbf x)\, dw(\mathbf x) \quad \text{for all } \varphi \in\mathcal S(\mathbb R^N), 
\end{equation}
where $\mathcal S(\mathbb R^N)$ denote the class of Schwartz functions on $\mathbb R^N$. 

\subsection{Dunkl kernel}\label{sec:Dunkl_transform}
For any fixed $\mathbf y\in\mathbb R^N$, the {\it Dunkl kernel} $\mathbf{x} \longmapsto E(\mathbf x,\mathbf y)$ is a unique analytic solution to the system 
$$D_\xi f=\langle \xi ,\mathbf y\rangle f, \quad f(0)=1.$$
The function $E(\mathbf x,\mathbf y)$, which generalizes the exponential function $e^{\langle\mathbf x,\mathbf y\rangle}$, has a unique extension to a holomorphic  function $E(\mathbf z,\mathbf w)$ on $\mathbb C^N\times \mathbb C^N$. 

\subsection{Dunkl transform}
Let $f \in L^1(dw)$. The \textit{Dunkl transform }$\mathcal{F}f$ of $f$ is defined  by
\begin{equation}\label{eq:Dunkl_transform}
    \mathcal{F} f(\xi)=\mathbf{c}_k^{-1}\int_{\mathbb{R}^N}f(\mathbf{x})E(\mathbf{x},-i\xi)\, {dw}(\mathbf{x}), \text{ where } \mathbf{c}_k=\int_{\mathbb{R}^N}e^{-\frac{{\|}\mathbf{x}{\|}^2}2}\,{dw}(\mathbf{x}){>0}.
\end{equation}
The Dunkl transform is a generalization of the Fourier transform on $\mathbb{R}^N$. It was introduced in~\cite{D5} for $k \geq 0$ and further studied in~\cite{dJ1} in the more general context. It has a lot of properties analogous to the classical Fourier transform, e.g., 
\begin{equation}\label{eq:T_j_transform}
    \mathcal{F}(D_{j}f)(\xi)=i\xi_{j}\mathcal{F}f(\xi) \text{ for all }f \in \mathcal{S}(\mathbb{R^N}) \text{ and }j \in \{1,\ldots,N\}.
\end{equation}
Moreover, it was proved in~\cite[Corollary 2.7]{D5} (see also~\cite[Theorem 4.26]{dJ1}) that it extends uniquely  to  an isometry on $L^2(dw)$.
Further,  the following inversion formula holds (\cite[Theorem 4.20]{dJ1}): for all $f \in L^1(dw)$ such that $\mathcal{F}f \in L^1(dw)$ one has  
$$f(\mathbf{x})=(\mathcal{F})^2f(-\mathbf{x}) \text{ for almost all }\mathbf{x} \in \mathbb{R}^N\textup{.}$$

\subsection{Dunkl translations}
Suppose that $f \in \mathcal{S}(\mathbb{R}^N)$.  We define the \textit{Dunkl translation }$\tau_{\mathbf{x}}f$ of $f$ to be
\begin{equation}\label{eq:translation}
    \tau_{\mathbf{x}} f(-\mathbf{y})=\mathbf{c}_k^{-1} \int_{\mathbb{R}^N}{E}(i\xi,\mathbf{x})\,{E}(-i\xi,\mathbf{y})\,\mathcal{F}f(\xi)\,{dw}(\xi)=\mathcal{F}^{-1}(E(i\cdot,\mathbf{x})\mathcal{F}f)(-\mathbf{y}).
\end{equation}
The Dunkl translation was introduced in~\cite{R1998}. The definition can be extended to functions which are not necessary in $\mathcal{S}(\mathbb{R}^N)$. For instance, thanks to the Plancherel's theorem, one can define the Dunkl translation of $L^2(dw)$ function $f$ by
\begin{equation}\label{eq:translation_Fourier}
    \tau_{\mathbf{x}}f(-\mathbf{y})=\mathcal{F}^{-1}(E(i\cdot,\mathbf{x})\mathcal{F}f(\cdot))(-\mathbf{y})
\end{equation}
(see~\cite{R1998} and~\cite[Definition 3.1]{ThangaveluXu}). In particular, the operators $f \mapsto \tau_{\mathbf{x}}f$ are contractions on $L^2(dw)$. Here and subsequently, for a reasonable function $g(\mathbf{x})$, we  write $g(\mathbf x,\mathbf y):=\tau_{\mathbf x}g(-\mathbf y)$. Moreover, it follows from~\eqref{eq:T_j_transform} and~\eqref{eq:translation} that for  $\varphi \in \mathcal{S}(\mathbb{R}^N)$, 
\begin{equation}\label{eq:translations_commute_with_operators}
    D_{j,\mathbf{x}}\{\varphi(\mathbf{x},\mathbf{y})\}=(D_j\varphi)(\mathbf{x},\mathbf{y})=-D_{j,\mathbf y} \varphi (\mathbf x,\mathbf y), \quad \mathbf{x},\mathbf{y} \in \mathbb{R}^N,  \ j=1,2,\ldots,N. 
\end{equation}

 The following specific formula  for the Dunkl translations of (reasonable) radial functions $f({\mathbf{x}})=\tilde{f}({\|\mathbf{x}\|})$ was obtained by R\"osler \cite{Roesler2003}:
\begin{equation}\label{eq:translation-radial}
\tau_{\mathbf{x}}f(-\mathbf{y})=\int_{\mathbb{R}^N}{(\tilde{f}\circ A)}(\mathbf{x},\mathbf{y},\eta)\,d\mu_{\mathbf{x}}(\eta)\text{ for }\mathbf{x},\mathbf{y}\in\mathbb{R}^N.
\end{equation}
Here
\begin{equation*}
A(\mathbf{x},\mathbf{y},\eta)=\sqrt{{\|}\mathbf{x}{\|}^2+{\|}\mathbf{y}{\|}^2-2\langle \mathbf{y},\eta\rangle}=\sqrt{{\|}\mathbf{x}{\|}^2-{\|}\eta{\|}^2+{\|}\mathbf{y}-\eta{\|}^2}
\end{equation*}
and $\mu_{\mathbf{x}}$ is a probability measure, 
which is supported in the set $\operatorname{conv}\mathcal{O}(\mathbf{x})$,  where $\mathcal O(\mathbf x) =\{\sigma(\mathbf x): \sigma \in G\}$ is the orbit of $\mathbf x$.

Formula~\eqref{eq:translation-radial} implies that for all radial $f \in L^1(dw)$ and $\mathbf{x} \in \mathbb{R}^N$, we have
\begin{align}\label{eq:tau_L1}
    \|\tau_{\mathbf{x}}f\|_{L^1(dw)} \leq \|f\|_{L^1(dw)}. 
\end{align}

\subsection{Dunkl convolution}
Assume that $f,g \in L^2(dw)$. The \textit{generalized convolution} (or the \textit{Dunkl convolution}) $f*g$ is defined by the formula
\begin{equation}\label{eq:conv_transform}
    f*g(\mathbf{x})=\mathbf{c}_k\mathcal{F}^{-1}\big((\mathcal{F}f)(\mathcal{F}g)\big)(\mathbf{x}),
\end{equation}
and equivalently, by
\begin{equation}\label{eq:conv_translation}
    (f*g)(\mathbf{x})=\int_{\mathbb{R}^N}f(\mathbf{y})\,\tau_{\mathbf{x}}g(-\mathbf{y})\,{dw}(\mathbf{y})=\int_{\mathbb{R}^N}g(\mathbf{y})\,\tau_{\mathbf{x}}f(-\mathbf{y})\,{dw}(\mathbf{y}).
\end{equation}
Generalized convolution of $f,g \in \mathcal{S}(\mathbb{R}^N)$ was considered in~\cite{R1998} and~\cite{Trimeche}, the definition was extended to $f,g \in L^2(dw)$ in~\cite{ThangaveluXu}. 

It follows from ~\eqref{eq:tau_L1} that if $f\in L^1(dw)$ is radial, then for any $g\in L^p(dw)$, $1\leq p\leq\infty$,  one has 
\begin{equation}
    \label{eq:L1-Lp}
    \| g*f\|_{L^p(dw)}\leq \| f\|_{L^1(dw)}\| g\|_{L^p(dw)}. 
\end{equation}

\subsection{Dunkl Laplacian, Dunkl heat semigroup, and Dunkl heat kernel}
The \textit{Dunkl Laplacian} associated with $R$ and $k$  is the differential-difference operator $\Delta_k=\sum_{j=1}^N D_{j}^2$. It was introduced in~\cite{Dunkl}, where it was also proved that $\Delta_k$ acts on $C^2(\mathbb{R}^N)$ functions by
\begin{equation}\label{eq:laplace_formula}
    \Delta_k f(\mathbf x)=\Delta f(\mathbf x)+\sum_{\alpha\in R} k(\alpha) \delta_\alpha f(\mathbf x), \text{ where }\delta_\alpha f(\mathbf x)=\frac{\partial_\alpha f(\mathbf x)}{\langle \alpha , \mathbf x\rangle} -  \frac{f(\mathbf x)-f(\sigma_\alpha (\mathbf x))}{\langle \alpha, \mathbf x\rangle^2}.
\end{equation}
Here $\Delta=\sum_{j=1}^{N}\partial_j^2$. It follows from~\eqref{eq:T_j_transform} that for all $\xi \in \mathbb{R}^N$ and $f \in \mathcal{S}(\mathbb{R}^N)$, we have
\begin{equation}\label{eq:Laplacian_on_Fourier_side}
    \mathcal{F}(\Delta_{k}f)(\xi)=-\|\xi\|^2\mathcal{F}f(\xi).
\end{equation} 
The operator $(-\Delta_{k},\mathcal{S}(\mathbb{R}^N))$ in $L^2(dw)$ is densily defined and closable. Its closure generates a strongly continuous and positivity-preserving contraction semigroup on $L^2(dw)$, which is given by
 \begin{equation}\label{eq:heat_semigroup}
  H_t f(\mathbf x)=\int_{\mathbb R^N} h_t(\mathbf x,\mathbf y)f(\mathbf y)\, dw(\mathbf y), 
  \end{equation}
where 
 \begin{equation}\label{eq:heat_kernel} \ h_t(\mathbf x,\mathbf y)=h_t(\mathbf y,\mathbf x)={\mathbf c}_k^{-1} (2t)^{-\mathbf N\slash 2}
 e^{-(\|\mathbf x\|^2+\|
 \mathbf y\|^2)\slash (4t)}E\left(\frac{\mathbf x}{\sqrt{2t}},\frac{\mathbf y}{\sqrt{2t}}\right)
 \end{equation}
 is so called  the \textit{generalized heat kernel} (or the \textit{Dunkl heat kernel}), see~\cite{R1998} .  The integral kernels $h_t(\mathbf x,\mathbf y)$ are the generalized translations of the Schwartz-class functions:
 $$ h_t(\mathbf x,\mathbf y)=\tau_{\mathbf x}h_t(-\mathbf y), \quad h_t(\mathbf x)={\mathbf c}_k^{-1} (2t)^{-\mathbf N/2} e^{-\|\mathbf x\|^2/4t}. $$ 
  The semigroup $\{H_t\}_{t>0}$ on $L^2(dw)$ can be expressed by means of the Dunkl transform, that is,  
\begin{equation}\label{eq:heat_transform}
    \mathcal{F}(H_tf)(\xi){=\mathcal F(h_t*f)(\xi)=\mathbf c_k \mathcal F(h_t)(\xi)\mathcal Ff(\xi)}=e^{-t\|\xi\|^2}\mathcal{F}f(\xi),\quad f\in L^2(dw).
\end{equation} 
 Note that in the case $k \equiv 0$ the Dunkl heat kernel is the classical heat kernel. Formula~\eqref{eq:heat_semigroup} defines  contraction semigroups on the $L^p(dw)$-spaces, $1\leq p\leq  \infty$, which are  strongly continuous for $1\leq p<\infty$.  
 
 \subsection{Estimates for generalized translations of some functions}
 
 For the purpose of this work we  need  bounds for the Dunkl heat kernel and its derivatives. 
 For $\mathbf{x},\mathbf{y} \in \mathbb{R}^N$, let 
\begin{equation}\label{eq:distance_of_orbits}
    d(\mathbf x,\mathbf y)=\min_{\sigma\in G}\| \sigma(\mathbf x)-\mathbf y\|    
\end{equation}
be the \textit{distance of the orbit} of $\mathbf{x}$ to the orbit of $\mathbf{y}$.
 For $\mathbf{x},\mathbf{y} \in \mathbb{R}^N$ and $t,r>0$, we denote
\begin{equation}\label{eq:V}
    V(\mathbf{x},\mathbf{y},r):=\max\{w(B(\mathbf{x},r)),w(B(\mathbf{y},r))\}, \ \ \mathcal G_t(\mathbf x,\mathbf y)=\frac{1}{V(\mathbf x,\mathbf y,\sqrt{t})}e^{-\frac {d(\mathbf x,\mathbf y)^2}{t}}.
\end{equation} 
It follows by the standard arguments, using the $G$-invariance of $w$,  that there is a constant $C>0$ such that for all $\mathbf{x} \in \mathbb{R}^N$ and $t>0$ we have
\begin{equation}\label{eq:G_integrable}
    \int_{\mathbb{R}^N}\mathcal G_t(\mathbf x,\mathbf y)\,dw(\mathbf{y}) \leq C.
\end{equation} 
The following theorem was proved in~\cite{DH-atom}, see also~\cite[Theorem 4.1]{ADzH}.
For more detailed upper and lower bounds for $h_t(\mathbf x,\mathbf y)$ we refer the reader to~\cite{DH-heat}.

\begin{theorem}[Theorem 4.1, \texorpdfstring{~\cite{DH-atom}}{[DH]}]\label{teo:heat_new}   For every nonnegative integer $m$ and for all multi-indices $\boldsymbol{\alpha},\boldsymbol{\beta} \in \mathbb{N}_0^N$ there are constants $C_{m,\boldsymbol{\alpha},\boldsymbol{\beta}}, c>0$ such that
  \begin{equation}\label{eq:heat2} |\partial_t^m \partial_{\mathbf x}^{\boldsymbol{\alpha}}\partial_{\mathbf y}^{\boldsymbol{\beta}} h_t(\mathbf{x},\mathbf{y})|
  \leq C_{m,\boldsymbol{\alpha},\boldsymbol{\beta}} t^{-m-\frac{|\boldsymbol{\alpha}|}{2}-\frac{|\boldsymbol{\beta}|}{2}} \Big(1+\frac{\| \mathbf x-\mathbf y\|}{\sqrt{t}}\Big)^{-2} \mathcal G_{t\slash c} (\mathbf x,\mathbf y).
  \end{equation}
  Moreover, if $\|\mathbf y-\mathbf y'\|\leq \sqrt{t}$, then
  \begin{equation}\label{eq:heat3}
  \begin{split}|\partial_t^m  h_t(\mathbf{x},\mathbf{y}) & -
 \partial_t^m  h_t(\mathbf{x},\mathbf{y'})|  \leq C_{m} t^{-m} \frac{\|\mathbf y-\mathbf y'\|}{\sqrt{t}}\Big(1+\frac{\| \mathbf x-\mathbf y\|}{\sqrt{t}}\Big)^{-2} \mathcal G_{t\slash c} (\mathbf x,\mathbf y).
  \end{split}\end{equation}
\end{theorem}
 
 We have the following estimate for generalized translations of Schwartz class functions.
 
\begin{theorem}[Theorem 4.1 and Remark 4.2,\texorpdfstring{~\cite{DzH_nonradial}}{[DzH]}]\label{teo:translacja}
     Let $\varphi \in \mathcal{S}(\mathbb{R}^N)$ and $M>0$. There is a constant $C>0$ such that for all $\mathbf{x},\mathbf{y} \in \mathbb{R}^N$ and $t>0$, we have 
   \begin{equation}\label{eq:g_t}
       |\varphi_t(\mathbf{x},\mathbf{y})| \leq C\left(1+\frac{\|\mathbf{x}-\mathbf{y}\|}{t}\right)^{-1}\Big(1+\frac{d(\mathbf x,\mathbf y)}{t}\Big)^{-M} \frac{1}{w(B(\mathbf{x},t))}, 
   \end{equation}
   where $  \varphi_t(\mathbf x)=t^{-\mathbf N}\varphi(\mathbf x/t)$.  
\end{theorem}
{The following corollary is a simple consequence of ~\eqref{eq:translations_commute_with_operators} and Theorem~\ref{teo:translacja}. 
\begin{corollary}
    \label{cor:simple_bound} For any non-negative integer $m$ and any multi-indices $\boldsymbol \beta, \boldsymbol \beta' \in \mathbb{N}_0^N$, and any $M>0$ there is a constant $C>0$ such that 
    \begin{equation}
         |\partial_t^m D_{\mathbf x}^{\boldsymbol \beta} D_{\mathbf y}^{ \boldsymbol \beta'} h_t(\mathbf x,\mathbf y)|\leq Ct^{-m-(|\boldsymbol \beta|+\boldsymbol \beta'|)/2}V(\mathbf x,\mathbf y,\sqrt{t})^{-1}\Big(1+\frac{\|\mathbf x-\mathbf y\|}{\sqrt{t}} \Big)^{-1} \Big(1+\frac{d(\mathbf x,\mathbf y)}{\sqrt{t}}\Big)^{-M}. 
    \end{equation} 
\end{corollary}
}
\subsection{\texorpdfstring{$k$}{k}-Cauchy kernel and Dunkl Poisson semigroup}

Let $\mathbf{x},\mathbf{y} \in \mathbb{R}^N$ and $t>0$. We define the $k$--\textit{Cauchy kernel} $p_t(\mathbf x,\mathbf y)$ to be the integral kernel of the operator $P_t=e^{-t\sqrt{-\Delta_{k}}}$.
It is related with the Dunkl heat kernel by the {subordination formula}
\begin{equation}\label{eq:subordination}\index{subordination formula}
p_t(\mathbf x,\mathbf y)= \Gamma(1/2)^{-1}\int_0^{\infty} e^{-u}h_{t^2\slash (4u)}(\mathbf x,\mathbf y) \frac{du}{\sqrt{u}}.
\end{equation}
Clearly,

\begin{equation}\label{eq:symmetric}
p_t(\mathbf x,\mathbf y)=p_t(\mathbf y,\mathbf x) \text{ for all }\mathbf{x},\mathbf{y} \in \mathbb{R}^N \text{ and }t>0,
\end{equation} 
{\begin{equation}\label{eq:one}
    \int_{\mathbb{R}^N}p_t(\mathbf{x},\mathbf{y})\, dw(\mathbf{y})=1 \text{ for all }\mathbf{x} \in \mathbb{R}^N \text{ and }t>0.
\end{equation}}
The kernel $p_t(\mathbf x,\mathbf y)$ was introduced and
studied in~\cite{RV1998}.

\begin{theorem}[{\cite[Theorem 5.6]{RV1998}}]
Let $f$ be a bounded continuous function on $\mathbb{R}^N$. Then the function given by $u(\mathbf{x},t)=P_tf(\mathbf{x})$ is continuous and bounded. Moreover, it solves the Cauchy problem
\begin{equation}\label{eq:Poisson_PDE}
\begin{cases}
    \partial_{t}^2u(\mathbf{x},t)+\Delta_{k,\mathbf{x}}u(\mathbf{x},t)=0 \text{ on }\mathbb{R}^N \times (0,\infty),\\
    u(\mathbf{x},0)=f(\mathbf{x}) \text{ for all }\mathbf{x} \in \mathbb{R}^N.
\end{cases}
\end{equation}
\end{theorem}
The $k$-Cauchy kernel is also called the \textit{generalized Poisson kernel} (or the \textit{Dunkl-Poisson kernel}) by the analogy with the classical Poisson semigroup. It follows from~\eqref{eq:heat_transform} and~\eqref{eq:subordination} that
\begin{equation}\label{eq:Poisson_transform}
    \mathcal{F}(P_tf)(\xi)=e^{-t\|\xi\|}\mathcal{F}f(\xi),\quad f\in \mathcal{S}(\mathbb{R}^N), \ \ \xi \in \mathbb{R}^N, \ \ t>0.
\end{equation} 
 We shall use the following bounds of the integral kernel 
$p_t(\mathbf x,\mathbf y)$ of the Dunkl-Poisson semigroup.

\begin{proposition}[{\cite[Proposition 5.1]{ADzH}}]\label{Poiss_new}
For any non-negative integer \,$m$ and for any multi-index \,$\boldsymbol{\beta} \in \mathbb{N}_0^{N}$, there is a constant \,$C\hspace{-.5mm}\ge\hspace{-.5mm}0$ such that, for all \,$t>0$ and for all \,$\mathbf{x},\mathbf{y}\in\mathbb{R}^N$,
\begin{equation}\label{DtDyPoisson}
\bigl|\hspace{.25mm}\partial_t^m\partial_{\mathbf{y}}^{\boldsymbol{\beta}}\hspace{.25mm}p_t(\mathbf{x},\mathbf{y})\bigr|\le C\,p_t(\mathbf{x},\mathbf{y})\hspace{.25mm}\bigl(\hspace{.25mm}t\hspace{-.25mm}+d(\mathbf{x},\mathbf{y})\bigr)^{\hspace{-.5mm}-m-|\boldsymbol{\beta}|}\times\begin{cases}
\,1&\text{if \,}m\hspace{-.25mm}=\hspace{-.25mm}0\hspace{.25mm},\\
\,1+\frac{d(\mathbf{x},\mathbf{y})}t&\text{if \,}m\hspace{-.5mm}>\hspace{-.5mm}0\hspace{.25mm}.\\
\end{cases}\end{equation}
Moreover, for any non-negative integer \,$m$ and for all multi-indices $\boldsymbol{\alpha},\boldsymbol{\beta}$, there is a constant \,$C\hspace{-.5mm}\ge\hspace{-.5mm}0$ such that  for all \,$t>0$ and for all \,$\mathbf{x},\mathbf{y}\in\mathbb{R}^N$,
\begin{equation}\label{DtDxDyPoisson}
\bigl|\hspace{.25mm}\partial_t^m\partial_{\mathbf{x}}^{\boldsymbol{\alpha}}\partial_{\mathbf{y}}^{\boldsymbol{\beta}}p_t(\mathbf{x},\mathbf{y})\bigr|\le C\,t^{-m-|\boldsymbol{\alpha}|-|\boldsymbol{\beta}|}\,p_t(\mathbf{x},\mathbf{y})\,.
\end{equation}
\end{proposition}

\begin{proposition}[{\cite[Proposition 3.6]{DH-atom}}]\label{propo:Poisson} There is a constant $C>0$ such that for all $t>0$, $\mathbf{x},\mathbf{y} \in \mathbb{R}^N$, we have
\begin{equation}
\label{eq:Poisson_new}
 p_t(\mathbf x,\mathbf y)\leq  C\frac{t}{V(\mathbf x,\mathbf y, d(\mathbf x,\mathbf y)+t)}\cdot \frac{d(\mathbf x,\mathbf y)+t}{\| \mathbf x-\mathbf y\|^2+t^2}
\end{equation}
if $N \geq 2$. If $N=1$, then
\begin{equation}
\label{eq:Poisson_dim1_new}
 p_t(\mathbf x,\mathbf y)\leq C \frac{t}{V(\mathbf x,\mathbf y, d(\mathbf x,\mathbf y)+t)}\cdot \frac{d(\mathbf x,\mathbf y)+t}{\| \mathbf x-\mathbf y\|^2+t^2}\cdot \ln\Big( 1+\frac{\| \mathbf x-\mathbf y\|+t}{d(\mathbf x,\mathbf y)+t}   \Big).
\end{equation}
  \end{proposition}

 Let us note that Proposition~\ref{propo:Poisson} implies that 
\begin{equation}
    \label{eq:Poisson_simple}
    p_t(\mathbf x,\mathbf y) \leq CV(\mathbf x,\mathbf y,t+d(\mathbf x,\mathbf y))^{-1} \Big(1+\frac{\|\mathbf x-\mathbf y\|}{t}\Big)^{-1}, \quad \mathbf x,\mathbf y\in\mathbb R^N, \ t>0. 
\end{equation}

Let us note that if $X=L^1(dw)$, then $p_t(\mathbf x,\mathbf y)$ is the integral kernel of the dual operator $P_t^*$ acting on $X^*=L^\infty$  {(see~\eqref{eq:symmetric}).}  Hence, we {sometimes} skip '$*$' and write $P_tf$ instead of $P_t^*f$ for $f\in L^\infty$. 

The following corollary of~\eqref{DtDxDyPoisson} will be used later on.

\begin{corollary}
    Let $\boldsymbol{\alpha}$ be a multi-index and let $m$ be a non-negative integer. There is a constant $C>0$ such that for all $f \in L^{\infty}$, $t>0$, and $\mathbf{x} \in \mathbb{R}^N$ we have:
    \begin{enumerate}[(a)]
        \item{the function 
        \begin{align*}
            { u(t,\mathbf x):}=P_tf(\mathbf x)=\int_{\mathbb{R}^N} p_t(\mathbf x,\mathbf y)f(\mathbf y)\, dw(\mathbf y)
        \end{align*}
is a $C^\infty$ as a function of the variables $(t,\mathbf x)\in (0,\infty)\times \mathbb R^N$;}
\item{
\begin{equation}
    \partial_t^m \partial_{\mathbf x}^{\boldsymbol{\alpha}}u(t,\mathbf x)=\int_{\mathbb{R}^N}  \Big\{\partial_t^m \partial_{\mathbf x}^{\boldsymbol{\alpha}} p_t(\mathbf x,\mathbf y)\Big\} f(\mathbf y)\, dw(\mathbf y); 
\end{equation}}
\item{\begin{equation}\label{eq:P_tfBounds}  | \partial_t^m \partial_{\mathbf x}^{\boldsymbol{\alpha}} u(t,\mathbf x)|\leq C_{m,\boldsymbol\alpha} t^{-m-|\boldsymbol{\alpha}|} \|f\|_{L^{\infty}}.  
\end{equation}}
    \end{enumerate}
\end{corollary}

 \section{Lipschitz spaces in the Dunkl setting}
\subsection{Preface} 

Consider the Dunkl heat semigroup $\{H_t\}_{t>0}$ acting on the Banach space $X=L^1(dw)$.  Theorem \ref{teo:heat_new} implies that $\frac{d}{dt} H_t$ is a bounded operator on $L^1(dw)$ and 
$$\Big\| \frac{d}{dt} H_t\Big\|_{L^1(dw)\to L^1(dw)}\leq C_1t^{-1}.$$ 
From a general theory of semigroups  (see e.g. Davies \cite{Davies}) we conclude that the semigroup $\{H_t\}_{t \geq 0}$ is holomorphic and uniformly bounded in a sector $\boldsymbol \Delta_\delta$ for some $\delta>0$.

Further, from  Corollary \ref{cor:simple_bound} and the fact that the Dunkl operators commute, we deduce that  \eqref{eq:A-form}--\eqref{eq:D_j_bounds} are satisfied for $\mathcal{T}_t=H_t$ and $\mathcal D_j=D_j$.

\begin{definition}
Consider $X=L^1(dw)$ and its dual $X^*=L^\infty=L^{\infty}(\mathbb R^N)$.    For $\beta>0$, we set $\Lambda_{k}^\beta=\Lambda^\beta_{-\sqrt{-\Delta_k}}$  and  $\|f\|_{\Lambda_k^\beta}=\| f\|_{\Lambda^\beta_{-\sqrt{-\Delta_k}}}$. 
\end{definition}

We are now in a position to state some results concerning $\Lambda^{\beta}_{k}$ which follows from the general approach described in  Part~\ref{part1} of the paper (see Theorems \ref{theo:Lamba=Lambda}, \ref{theo:Bessel}, \ref{teo:A-special}). 

\begin{theorem}\label{teo:Lip_Dunk_easy}
\begin{enumerate}[(a)]
    \item{$\Lambda_{k}^{2\beta}=\Lambda_{\Delta_k}^\beta$, $\beta>0$. }
    \item{For $\beta,\gamma>0$, the Dunkl Bessel potential \textup{(}see \eqref{eq:Bessel}\textup{)} 
  $f\mapsto ((I-\Delta_k)^{-\gamma/2})^*f=f* \mathcal J^{\gamma/2},$ 
 is an isomorphism  of $\Lambda_{k}^{\beta}$ onto $\Lambda_{k}^{\beta+\gamma}$. 
  Here  
 $$ \mathcal J^{\beta/2} (\mathbf x)={\Gamma (\beta/2)^{-1}}\int_0^\infty t^{\beta/2} e^{-t}h_t(\mathbf x )\frac{dt}{t}$$ 
 is a radial $L^1(dw)$-function. }\label{numitem:easy_b}
 \item{Suppose $\beta>1$ and $f\in L^\infty$.  Then $f\in \Lambda_{k}^\beta$ if and only if $D_j^*f\in \Lambda_{k}^{\beta-1}$. Moreover,  {there is a constant $C>1$ such that for all $f \in \Lambda_k^{\beta}$, we have}
 $$ C^{-1}\| f\|_{\Lambda_{k}^\beta} \leq \| f\|_{L^\infty} +\sum_{j=1}^N \| D^*_jf\|_{\Lambda_{k}^{\beta-1}} \leq C\| f\|_{\Lambda_{k}^\beta} .$$
 Here $D_j^*f$ is understood in the mild sense \textup{(}see \eqref{eq:weak-sense}\textup{)}. }\label{numitem:easy_c}
\end{enumerate}
\end{theorem}

 A remark is in order. One can easily prove that the action $D_j^*f$ for the characterization  stated in part~\eqref{numitem:easy_c} of the theorem can be equivalently understood in the following distributional sense, namely $g_j=D_j^*f$ if and only if  
 $$ \int_{\mathbb R^N} g_j(\mathbf x)\mathbf \varphi (\mathbf x)\, dw(\mathbf x)=-\int_{\mathbf R^N} f(\mathbf x)D_j\varphi(\mathbf x)\, dw(\mathbf x) \quad 
 \text{ for all }\varphi\in C_c^\infty(\mathbb R^N).$$ 

\subsection{Case $0<\beta<1$ }
\begin{theorem}\label{teo:Lambda}
    For $0<\beta<1$, the spaces $\Lambda^\beta(\mathbb R^N)$ and $\Lambda_{k}^\beta$ coincide {and the corresponding norms $\| \cdot\|_{\Lambda^\beta(\mathbb R^N)}$ and  $\| \cdot\|_{\Lambda^\beta_k}$ are equivalent. }
\end{theorem}

 We start by the lemma. 

 \begin{lemma}\label{lem:Lipschitz_P}
     There is a constant $C>0$ such that for all $f\in L^\infty$, $\mathbf x,\mathbf x'\in \mathbb R^N $, and $t>0$ we have 
     $$ |P_tf(\mathbf x)-P_tf(\mathbf x')|\leq C \min\Big(1,\frac{\|\mathbf x-\mathbf x'\|}{t}\Big)\| f\|_{L^\infty}.$$
 \end{lemma}
 \begin{proof} Thanks to~\eqref{eq:P_tfBounds} it suffices to consider $\|\mathbf x-\mathbf x'\|<t$. Then 
\begin{equation}
    \begin{split}\label{eq:Lip_int}
        \int_{\mathbb R^N} |p_t(\mathbf x,\mathbf y)-p_t(\mathbf x',\mathbf y)|\, dw(\mathbf y)&=\int_{\mathbb R^N} \Big|\int_0^1 \partial_s p_t(\mathbf x'+s(\mathbf x-\mathbf x'),\mathbf y)\, ds \Big|\, dw(\mathbf y)\\
        & \leq \int_{\mathbb R^N} \int_0^1  \|\mathbf x-\mathbf x'\| \| \nabla_{\mathbf{x}} p_t(\mathbf x'+s(\mathbf x-\mathbf x'),\mathbf y)\|\, ds \, dw(\mathbf y)\\
      &  \leq Ct^{-1}\|\mathbf x-\mathbf x'\|,
    \end{split}
\end{equation}
where in the last inequality we have used  ~\eqref{DtDxDyPoisson} (with $m=0$, $|\boldsymbol\alpha|=1$, and $\boldsymbol \beta=\boldsymbol 0$) together with ~\eqref{eq:Poisson_simple}. Now the lemma follows from ~\eqref{eq:Lip_int}.
\end{proof}

\begin{proof}[Proof of Theorem \ref{teo:Lambda}] 

For $\mathbf{x},\mathbf{y} \in \mathbb{R}^N$ and $t>0$, we set 
\begin{equation*}
   q_t(\mathbf x):=\frac{d}{dt} p_t(\mathbf x), \quad  q_t(\mathbf x,\mathbf y):=\frac{d}{dt} p_t(\mathbf x,\mathbf y).
\end{equation*}
 Then, by~\eqref{eq:symmetric} and~\eqref{eq:one},  $q_t(\mathbf x,\mathbf y)=q_t(\mathbf y,\mathbf x)$ and 
$$ Q_t^*f(\mathbf x):= \frac{d}{dt} P_t^*f(\mathbf x)=\int_{\mathbb R^N} q_t(\mathbf x,\mathbf y)f(\mathbf y)\, dw(\mathbf y),$$
\begin{equation}\label{eq:Fourier_qt}
    \quad   \int_{\mathbb{R}^N} q_t(\mathbf x,\mathbf y)\, dw(\mathbf y)=0.
\end{equation}

     Let $0<\beta<1$. Suppose $f\in \Lambda^\beta (\mathbb R^N)$. 
    Applying  ~\eqref{eq:Fourier_qt} together with \eqref{DtDxDyPoisson} combined with {Proposition~\ref{propo:Poisson}} with $m=1$ and $\boldsymbol \beta =\boldsymbol 0$, we have  
    \begin{equation*}
        \begin{split}
            \Big| \frac{d}{dt} P_t^*f(\mathbf x)\Big| & =\Big| \int_{\mathbb R^N} q_t(\mathbf x,\mathbf y) (f(\mathbf y)-f(\mathbf x))\, dw(\mathbf y)\Big|\\
            &\leq \|f\|_{\Lambda^\beta(\mathbb R^N)} \int_{\mathbb R^N}  |q_t(\mathbf x,\mathbf y)| \|\mathbf x-\mathbf y\|^\beta\, dw(\mathbf y)\\
            &\leq C \|f\|_{\Lambda^\beta(\mathbb R^N)} \int_{\mathbb R^N} t^{\beta -1}V(\mathbf x,\mathbf y, t+d(\mathbf x,\mathbf y))^{-1} \Big(1+\frac{\|\mathbf x-\mathbf y\|}{t}\Big)^{-1}  \frac{\|\mathbf x-\mathbf y\|^\beta }{t^\beta}\, dw(\mathbf y)\\
            &\leq C \|f\|_{\Lambda^\beta(\mathbb R^N)}  t^{\beta -1}.
        \end{split}
    \end{equation*}
    Hence, $\| f\|_{\Lambda^\beta_k}\leq C\| f\|_{\Lambda^\beta(\mathbb R^N)}$. 
    
The converse implication will be proven if we show that $\Lambda^\beta_k\subseteq \Lambda^\beta(\mathbb R^N)$ and there is a~constant $C>0$ such that for all $f \in \Lambda_k^\beta$ we have
\begin{equation}\label{eq:Lambda_f} \|f\|_{\Lambda^\beta(\mathbb R^N)}\leq C\| f\|_{\Lambda^\beta_k}. 
\end{equation}
       For this purpose, let us note that by Lemma \ref{lem:converge}, the functions $ P_t^*f$ (which are continuous on $\mathbb R^N$) converge uniformly to $f$, as $t$ tends to $0$.   Hence, $f$ is  function on $\mathbb R^N$. 
      We now turn to prove the Lipschitz regularity \eqref{eq:Lip_class} of $f$ and~\eqref{eq:Lambda_f}.  Fix   $\mathbf x,\mathbf x'\in \mathbb R^N$ such that  $0<\|\mathbf x-\mathbf x'\|\leq 1$. Set $\delta:=\|\mathbf x-\mathbf x'\|$. Then  
      \begin{equation}\label{eq:Lambda_split}
          \begin{split}
              |f(\mathbf x')-f(\mathbf x)|&=\big|\lim_{\varepsilon \to 0} (P^*_\varepsilon f(\mathbf x')-P^*_\varepsilon f(\mathbf x))\big|\\
              &{\leq \lim_{\varepsilon \to 0} \Big| \int_\varepsilon^1 \frac{d}{ds} \Big(P^*_sf(\mathbf x')-P^*_sf(\mathbf x)\Big)\, ds \Big|+\Big| P^*_1f(\mathbf x')-P^*_1 f(\mathbf x)\Big|}\\
              & \leq \lim_{\varepsilon\to 0} \int_\varepsilon^\delta \Big(\Big|\frac{d}{ds} P^*_sf(\mathbf x)\Big|\, + \Big|\frac{d}{ds} P^*_sf(\mathbf x')\Big|\Big)\, ds \\
              &\ \ +\int_\delta^1  \Big|  \frac{d}{ds} \big(P_sf(\mathbf x')
              -P^*_sf(\mathbf x)\big)\Big|\, ds
              + C\|\mathbf x-\mathbf x'\| \|f\|_{L^\infty}, 
          \end{split}
      \end{equation}
     where in the last inequality we have used Lemma~\ref{lem:Lipschitz_P}. {Since $f\in \Lambda_k^\beta$,} we have 
      \begin{equation}\label{eq:Lambda_1}
           \int_\varepsilon^\delta \Big(\Big|\frac{d}{ds} P^*_sf(\mathbf x)\Big| + \Big|\frac{d}{ds} P^*_sf(\mathbf x')\Big|\Big)\,ds\leq 2\| f\|_{\Lambda_k^\beta}\int_\varepsilon^\delta s^{\beta-1}\, ds \leq C \| f\|_{\Lambda_k^\beta}\|\mathbf x-\mathbf x'\|^\beta. 
      \end{equation}
Recall that, by the semigroup property,  $\frac{d}{ds}P^*_s=Q_s=P^*_{s/2}Q^*_{s/2}$. Hence, using Lemma~\ref{lem:Lipschitz_P} and the assumption $f\in\Lambda_k^\beta$, we get 
\begin{align*}
   & \int_{\delta}^1 \Big|\frac{d}{ds}(P^*_sf(\mathbf{x}')-P^*_sf(\mathbf{x}))\Big|\,ds=\int_\delta^1 \Big| P^*_{s/2}(Q^*_{s/2}f)(\mathbf{x}')-P^*_{s/2}(Q^*_{s/2}f)(\mathbf{x})\Big|\,ds \\&\leq C\int_\delta^1 \frac{\|\mathbf{x}'-\mathbf{x}\|}{s}\|Q^*_{s/2}f\|_{{L^\infty}}\,ds \leq C'\int_\delta^1 \frac{\|\mathbf{x}'-\mathbf{x}\|}{s}  s^{\beta-1} \| f\|_{\Lambda_k^\beta}\,ds \leq C''\|\mathbf{x}'-\mathbf{x}\|^\beta \| f\|_{\Lambda_k^\beta}.
\end{align*}
 If $\|\mathbf x-\mathbf x'\|>1$, then 
    $  |f(\mathbf x)-f(\mathbf x')|\leq 2\|f\|_{L^\infty}\leq 2\|\mathbf x-\mathbf x'\|^\beta \|f\|_{L^\infty}$. 
Thus the proof of the theorem is complete. 
\end{proof} 

\begin{remark}\normalfont
    {  For $0<\beta<1$ the homogeneous Lipschitz spaces 
$$ \dot\Lambda^{\beta} =\Big\{f:\mathbb R^N\to \mathbb C: \sup_{\mathbf x\ne\mathbf y} \frac{|f(\mathbf x)-f(\mathbf{y})|}{\|\mathbf x-\mathbf y\|^\beta}<\infty\Big\} $$
 associated with the Euclidean metric and the spaces 
 $$ \dot\Lambda^{\beta}_d =\Big\{f:\mathbb R^N\to \mathbb C: \sup_{\mathbf x\ne\mathbf y} \frac{|f(\mathbf x)-f(\mathbf{y})|}{d(\mathbf x,\mathbf y)^\beta}<\infty\Big\} $$ 
 related to the orbit distance $d(\mathbf x,\mathbf y)$ (see ~\eqref{eq:distance_of_orbits}) where studied in~\cite{Han_et_al}. One of the results of~\cite{Han_et_al} asserts that the space $\dot\Lambda^\beta$ is characterized by the condition 
 $$ \sup_{B}\frac{1}{\text{\rm diam} (B)^\beta} \Big(\frac{1}{w(B)}\int_B |f(\mathbf x)-f_B|^qdw(\mathbf x)\Big)^{1/q}<\infty $$
 for any/all  $1\leq q<\infty$, where the supremum is taken over all balls $B\subset \mathbb R^N$ and $f_B=w(B)^{-1}\int_B f\, dw$.  We also refer the reader to \cite{Han_et_al} for results which relates  the Triebel-Lizorkin spaces in the Dunkl setting and the commutators of Lipschitz functions $b$ with the Dunkl-Riesz transforms or the Dunkl-Riesz potentials $\Delta_k^{-\gamma}$.  }
\end{remark}

\section{Proof of Theorem \ref{teo:main_main}} 
\noindent 
The following theorem, which follows by  an iteration of \eqref{eq:step-less}, will be used in the proof of Theorem \ref{teo:main_main}.

\begin{theorem}[{\cite[Theorem 10]{Taibleson1}}]\label{teo:Lambda_it}
    Let $\beta>0$, $n\in\mathbb N_0$, $0\leq n<\beta$. Then $f\in \Lambda^\beta(\mathbb R^N)$    
    if and only if $f\in C^n(\mathbb R^N)$, $\partial^{\boldsymbol \gamma} f$  
    are bounded functions for $|\boldsymbol\gamma|\leq n$, and $\partial^{\boldsymbol\gamma} f\in \Lambda^{\beta-n}(\mathbb R^N)$ for all $\boldsymbol\gamma \in \mathbb{N}_0^N$ such that $|\boldsymbol\gamma|=n$.  Moreover, there is a constant $C>0$ such that for all $f \in \Lambda^\beta(\mathbb R^N)$ we have
    \begin{equation}
        \label{eq:equiv_norm3}
        C^{-1}\| f\|_{\Lambda^\gamma (\mathbb R^N)} \leq \sum_{|\boldsymbol\gamma|<n} \| \partial^{\boldsymbol\gamma} f \|_{L^\infty} +\sum_{|\boldsymbol\gamma|=n} \| \partial^{\boldsymbol\gamma} f \|_{\Lambda^{\beta-n}(\mathbb R^N)} \leq C\| f\|_{\Lambda^\gamma (\mathbb R^N)}. 
    \end{equation}
\end{theorem}

\begin{theorem}\label{teo:first_inclusion}
    Let $\beta>0$, $\beta \not\in \mathbb{N}$. Then $\Lambda^\beta (\mathbb R^N)\subseteq \Lambda^\beta_k$ {  and  there is a constant $C>0$ such that for all $f \in \Lambda^\beta(\mathbb R^N)$ we have $\| f\|_{\Lambda_k^\beta}\leq C \| f\|_{\Lambda^\beta(\mathbb R^N)}$.  }
\end{theorem}

\begin{proof}

Assume that $0\leq n<\beta<n+1$. The proof of the inclusion goes by induction on $n$. If $n=0$, the equality $\Lambda^\beta(\mathbb R^N)=\Lambda^\beta_k$ and {  the equivalence of the norms hold} thanks to Theorem~\ref{teo:Lambda}.  {  Suppose that $1\leq n<\beta<n+1$ and our induction  hypothesis  holds for $\beta-1<n$. Let } $f\in\Lambda^\beta (\mathbb R^N)$. 
Then $f\in C^n(\mathbb R^N)$ and $\partial_j f\in \Lambda^{\beta -1}(\mathbb R^N)$ for  $j=1,\ldots,N$. According to part~\eqref{numitem:easy_c} of  Theorem \ref{teo:Lip_Dunk_easy} combined with the the induction hypothesis, it suffices to prove that $D_j^*f \in \Lambda^{\beta-1}(\mathbb R^N)$ for $j=1,\ldots,N$. Since $f\in C^n(\mathbb R^N)$, $n\geq 1$, $D_j^*f=-D_jf$, where, by the definition of $D_j$, 
\begin{equation}
    \begin{split}
          D_jf(\mathbf{x})&=\partial_j f(\mathbf{x})+\sum_{\alpha\in R} \frac{k(\alpha)}{2} \alpha_j \frac{f(\mathbf{x})-f(\sigma_\alpha (\mathbf{x}))}{\langle \alpha,\mathbf{x}\rangle}\\
          &= \partial_j f(\mathbf{x}) + \sum_{\alpha\in R} \frac{k(\alpha)}{2}\alpha_j \int_0^1 \langle \alpha, \nabla f(\sigma_\alpha(\mathbf x)+s(\mathbf x-\sigma_\alpha(\mathbf x))\rangle \, ds.
    \end{split}
\end{equation}
Since $\partial_j f\in \Lambda^{\beta-1}(\mathbb R^N)$ {  and  $ \|\partial_j f\|_{\Lambda^{\beta-1}(\mathbb R^N)}\leq C\| f\|_{\Lambda^\beta(\mathbb R^N)}$,  }   it suffices to deal with the functions 
\begin{equation}\label{eq:g_alpha}
    \begin{split}
      f_{\alpha, \ell}(\mathbf x) :=  \int_0^1 ( \partial_\ell f)(\sigma_\alpha(\mathbf x)+s(\mathbf x-\sigma_\alpha(\mathbf x)) \, ds =\int_0^1 (\partial_\ell f)(A_{s,\alpha}(\mathbf x))\, ds,
    \end{split}
\end{equation}
for $\ell=1,2,...,N$, where $A_{s,\alpha}(\mathbf x):= \sigma_\alpha(\mathbf x)+s(\mathbf x-\sigma_\alpha(\mathbf x))$. 
Now the $C^{n-1}(\mathbb{R}^N)$-function $f_{\alpha,\ell} $ belongs to $\Lambda^{\beta-1}(\mathbb R^N)$ if and only if 
\begin{equation}\label{eq:norm_equivalence}
    \| f_{\alpha, \ell} \|_{\Lambda^{\beta-1}(\mathbb R^N)} \sim \sum_{|{\boldsymbol\gamma}|<n-1} \| \partial^{\boldsymbol\gamma} f_{\alpha,\ell}\|_{L^\infty} +\sum_{|{\boldsymbol\gamma}|=n-1}\| \partial^{\boldsymbol\gamma} f_{\alpha,\ell}\|_{\Lambda^{\beta -n}(\mathbb R^N)}<\infty. 
\end{equation}
 Note that the linear mappings $A_{s,\alpha}$ on the Euclidean space $\mathbb R^N$ satisfy 
$$ \|A_{s,\alpha}\|= 1, \quad 0\leq s\leq 1, \ \alpha \in R. $$
Thus, by the chain rule, we obtain
\begin{equation*}
    |\partial^{\boldsymbol\gamma} (\{ \partial_\ell f\} \circ A_{s,\alpha})(\mathbf x)|\leq C_{\boldsymbol{\gamma}} \| f\|_{C^n(\mathbb{R}^N)} \leq C'_{\boldsymbol{\gamma}} \|f\|_{\Lambda^{\beta}(\mathbb R^N)} \quad \text{for} \ |{\boldsymbol \gamma}| \leq n-1 
\end{equation*}
and 
\begin{equation*}
    |\partial^{\boldsymbol \gamma} ( \{ \partial_\ell f\} \circ A_{s,\alpha})(\mathbf x)- \partial^{\boldsymbol\gamma} (\{ \partial_\ell f\} \circ A_{s,\alpha})(\mathbf x')|\leq C \| f\|_{\Lambda^\beta(\mathbb R^N)}\|\mathbf x-\mathbf x'\|^{\beta-n}  \quad \text{for} \ |{\boldsymbol \gamma}|= n-1.
\end{equation*} 
Hence from~\eqref{eq:g_alpha}, we conclude  ~\eqref{eq:norm_equivalence}.  Consequently  $D_jf \in \Lambda^{\beta-1}(\mathbb R^N)$ {  and  $\|D_jf\|_{\Lambda^{\beta-1}(\mathbb R^N)}\leq C\| f\|_{\Lambda^\beta(\mathbb R^N)}$.} {Thus, by Theorem \ref{teo:Lip_Dunk_easy},  $f\in \Lambda_k^\beta$ and 
$\| f\|_{\Lambda_k^\beta}\leq C \| f\|_{\Lambda^\beta(\mathbb R^N)}$. }
\end{proof}

\begin{lemma}\label{lem:1}
    Let $\beta \in (0,1)$ and let $m$ be a positive integer. There are positive constants $C_{m,\beta}, \, C_{m,\beta}'$ such that for all $\boldsymbol{\alpha} \in \mathbb{N}^{N}_0$ such that $|\boldsymbol{\alpha}|=m$, $f \in \Lambda_{k}^{\beta}$,  $\mathbf{x},\mathbf{x}' \in \mathbb{R}$, and $t>0$ such that $\|\mathbf{x}-\mathbf{x}'\|<t$, we have
    \begin{equation}\label{eq:t_1_lem}
        \|\partial^{\boldsymbol{\alpha}}P_tf\|_{L^{\infty}} \leq C_{m,\beta} t^{\beta-m}\|f\|_{\Lambda^{\beta}_k},
    \end{equation}
    \begin{equation}\label{eq:t_2_lem}
        |\partial^{\boldsymbol{\alpha}}P_tf(\mathbf{x})-\partial^{\boldsymbol{\alpha}}P_tf(\mathbf{x}')| \leq C'_{m,\beta} \|\mathbf{x}-\mathbf{x}'\|t^{\beta-m-1}\|f\|_{\Lambda^{\beta}_k}.
    \end{equation}
\end{lemma}

\begin{proof}
    {In virtue of~\eqref{eq:P_tfBounds}, it } is enough to check~\eqref{eq:t_1_lem} for $0<t<1$. For $f \in \Lambda_k^\beta$ and $0<t<1$ we have 
    \begin{equation*}
        P_tf=P_1f-\int_t^1 \partial_s P_sf\,ds.
    \end{equation*}
By{~\eqref{eq:P_tfBounds}}, $\|\partial^{\boldsymbol{\alpha}}P_1f\|_{L^{\infty}} \leq C_{\boldsymbol\alpha}\|f\|_{L^\infty}$, so it remains to estimate
\begin{equation}\label{eq:dalpha}
    \partial^{\boldsymbol{\alpha}}\int_t^1 \partial_s P_sf\,ds=\int_t^1 \partial^{\boldsymbol{\alpha}}\partial_s P_sf\,ds.
\end{equation}
Further, 
\begin{equation}\label{eq:innner_outer}
    \partial^{\boldsymbol{\alpha}}\partial_s P_sf(\mathbf{x})=\int_{\mathbb{R}^N}\partial^{\boldsymbol{\alpha}}p_{s/2}(\mathbf{x},\mathbf{z})\int_{\mathbb{R}^N}q_{s/2}(\mathbf{z},\mathbf{y})f(\mathbf{y})\,dw(\mathbf{y})\,dw(\mathbf{z}).
\end{equation}
Since $f \in \Lambda_k^{\beta}$, there is $C>0$ such that for all $s>0$ and $\mathbf{z} \in \mathbb{R}^N$,  we have
\begin{equation}\label{eq:inner_1}
    \left|\int_{\mathbb{R}^N}q_{s/2}(\mathbf{z},\mathbf{y})f(\mathbf{y})\,dw(\mathbf{y}) \right| \leq C s^{\beta-1} \|f\|_{\Lambda_k^{\beta}}.
\end{equation}
From{~\eqref{eq:P_tfBounds}} we conclude that there is $C_{\boldsymbol\alpha}>0$ such that for all $s >0$ and $\mathbf{x} \in \mathbb{R}^N$,  we have 
\begin{equation}\label{eq:outer_1}
    \int_{\mathbb{R}^N}|\partial^{\boldsymbol{\alpha}}p_{s/2}(\mathbf{x},\mathbf{z})|\,dw(\mathbf{z}) \leq C_{\boldsymbol\alpha}s^{-m}.
\end{equation}
Combining~\eqref{eq:innner_outer},~\eqref{eq:inner_1}, and~\eqref{eq:outer_1} together with \eqref{eq:dalpha}, we get~\eqref{eq:t_1_lem}. Finally,~\eqref{eq:t_2_lem} is a consequence of~\eqref{eq:t_1_lem} and the mean value theorem. 
\end{proof}

\begin{lemma}\label{lem:3}
    Let $m \in \mathbb{N}$. There is a constant $C>0$ such that for all $\boldsymbol{\alpha} \in \mathbb{N}^{N}_0$ such that $|\boldsymbol{\alpha}|=m$, $f \in L^{\infty}$, $\mathbf{x},\mathbf{x}' \in \mathbb{R}$, and $t>0$ such that $\|\mathbf{x}-\mathbf{x}'\|<t$, we have
    \begin{equation}\label{eq:t_3_lem}
        |\partial^{\boldsymbol{\alpha}}P_tf(\mathbf{x})-\partial^{\boldsymbol{\alpha}}P_tf(\mathbf{x}')| \leq C \|\mathbf{x}-\mathbf{x}'\|t^{-m-1}\|f\|_{L^{\infty}}.
    \end{equation}
\end{lemma}

\begin{proof}
    The lemma follows directly from the estimates~\eqref{eq:P_tfBounds}.
\end{proof}

\begin{theorem}\label{teo:inclusion_2}
 Let $\gamma>1$, $\gamma \not\in \mathbb{N}$. We have $\Lambda_k^{\gamma} \subseteq \Lambda^{\gamma}(\mathbb R^N)$ {   and 
 \begin{equation}\label{eq:control1}
 \|g\|_{\Lambda^\gamma (\mathbb R^N)}\leq C_{\gamma} \| g\|_{\Lambda_k^\gamma}.
 \end{equation} }
\end{theorem}
\begin{proof}
    Fix  $\gamma>1$, $\gamma\notin \mathbb N$.  Let $m$ be such that $m<\gamma<m+1$. We fix $\gamma_1,\gamma_2>0$ such that $\gamma_1 \in (0,1)$ and $m<\gamma_2<m+1$, {  $\gamma_1+\gamma_2=\gamma$. Consider}  $g \in \Lambda_k^{\gamma}$. Our goal is to  verify    \eqref{eq:control1}. For this purpose, thanks to Theorem \ref{teo:Lambda_it}, it is enough to show that 
    \begin{equation}\label{eq:J_goal_1}
         \sum_{|\boldsymbol \alpha|<m}\| \partial^{\boldsymbol\alpha} g\|_{L^\infty} + \sum_{|\boldsymbol \alpha|=m}\| \partial^{\boldsymbol\alpha} g\|_{\Lambda^{\gamma-m}(\mathbb R^N)} \leq C\|g\|_{\Lambda_k^{\gamma}}.
    \end{equation}
   {    Let us recall that $\mathcal J_k^{\gamma_2}$ is an isomorphism of the space $\Lambda_k^{\gamma_1}$ onto  $\Lambda_k^{\gamma}$ (see Theorem \ref{teo:Lip_Dunk_easy}).} 
     Let $f \in \Lambda_k^{\gamma_1}$ be such that $J^{\gamma_2}_kf=g$. {  Clearly, $\| g\|_{\Lambda_k^\beta}\sim \| f\|_{\Lambda_k^{\gamma_1}}$. } Hence, instead of~\eqref{eq:J_goal_1},  it is enough to verify that 
    \begin{equation}\label{eq:J_goal_2}
         \sum_{|\boldsymbol \alpha|<m}\| \partial^{\boldsymbol\alpha}\mathcal J^{\gamma_2}_kf\|_{L^\infty} + \sum_{|\boldsymbol \alpha|=m}\| \partial^{\boldsymbol\alpha} \mathcal J^{\gamma_2}_kf\|_{\Lambda^{\gamma-m}(\mathbb R^N)} \leq C\|f\|_{\Lambda_k^{\gamma_1}}.
    \end{equation}
    First, let us take $\boldsymbol{\alpha} \in \mathbb{N}_0^{N}$ such that $|\boldsymbol{\alpha}|{\leq} m$ and estimate $\| \partial^{\boldsymbol\alpha}J^{\gamma_2}_kf\|_{L^\infty}$. By Lemma~\ref{lem:1}, we get
    \begin{equation}\label{eq:I_444}
    \begin{split}
        &\| \partial^{\boldsymbol\alpha}\mathcal J^{\gamma_2}_kf\|_{L^\infty} \leq \Gamma(\gamma_2)^{-1}\int_{0}^{\infty}e^{-t}t^{\gamma_2}\|\partial^{\boldsymbol{\alpha}}P_tf\|_{L^{\infty}}\frac{dt}{t} \\&\leq C\|f\|_{\Lambda_k^{\gamma_1}}\int_{0}^{\infty}e^{-t}t^{\gamma_2}t^{\gamma_1-|\boldsymbol{\alpha}|}\frac{dt}{t} \leq C\|f\|_{\Lambda_k^{\gamma_1}}.
        \end{split}
    \end{equation}
    Then, let  $\boldsymbol{\alpha} \in \mathbb{N}_0^{N}$ be such that $|\boldsymbol{\alpha}|{  =}m$ and estimate $\| \partial^{\boldsymbol\alpha} \mathcal J^{\gamma_2}_kf\|_{\Lambda^{\gamma-m}}$. Since $\gamma-m \in (0,1)$, it is enough to {verify the following  Lipschitz condition:  
    \begin{equation}\label{eq:Lip_gamma-m}
        \Big|\partial^{\boldsymbol\alpha} \mathcal J^{\gamma_2}_kf(\mathbf{x})-\partial^{\boldsymbol\alpha} \mathcal J^{\gamma_2}_kf(\mathbf{x}')\Big|\leq C\| f\|_{\Lambda_k^{\gamma_1}}\| \mathbf x-\mathbf x'\|^{\gamma-m}.
    \end{equation}
     For $\mathbf{x},\mathbf{x}' \in \mathbb{R}^N$  such that $\| \mathbf x-\mathbf x'\|\leq 1$, we write }
    \begin{align*}
        &\Big|\partial^{\boldsymbol\alpha} \mathcal J^{\gamma_2}_kf(\mathbf{x})-\partial^{\boldsymbol\alpha} \mathcal J^{\gamma_2}_kf(\mathbf{x}')\Big| \leq \Gamma(\gamma_2)^{-1}\int_0^{\infty}e^{-t} t ^{\gamma_2}\Big|\partial^{\boldsymbol{\alpha}}P_tf(\mathbf{x})-\partial^{\boldsymbol{\alpha}}P_tf(\mathbf{x}')\Big|  \frac{dt}{t}\\& \leq C\int_0^{\|\mathbf{x}-\mathbf{x}'\|}\cdots +C\int_{\|\mathbf{x}-\mathbf{x}'\|}^{1}\cdots +C\int_1^\infty=:I_1+I_2+I_3.
    \end{align*}
    From Lemma~\ref{lem:1} we get
    \begin{equation}\label{eq:I_111}
        \begin{split}
        I_1 &\leq C\int_{0}^{\|\mathbf{x}-\mathbf{x}'\|}e^{-t}t^{\gamma_2}\|\partial^{\boldsymbol{\alpha}}P_tf\|_{L^{\infty}}\frac{dt}{t} \leq  C\|f\|_{\Lambda_k^{\gamma_1}}\int_{0}^{\|\mathbf{x}-\mathbf{x}'\|}e^{-t}t^{\gamma_2}t^{\gamma_1-m}\frac{dt}{t}\\& \leq C\|f\|_{\Lambda_k^{\gamma_1}}\|\mathbf{x}-\mathbf{x}'\|^{\gamma -m }.
        \end{split}
    \end{equation}
    Further, using  \eqref{eq:t_2_lem}, we obtain
    \begin{equation}\label{eq:I_222}
        \begin{split}
            I_2 &\leq C\|\mathbf{x}-\mathbf{x}'\|\|f\|_{\Lambda_k^{\gamma_1}}\int_{\|\mathbf{x}-\mathbf{x}'\|}^{1}e^{-t}t^{\gamma_2}t^{\gamma_1-m-1}\frac{dt}{t} \\&\leq C\|\mathbf{x}-\mathbf{x}'\|\|f\|_{\Lambda_k^{\gamma_1}}\|\mathbf{x}-\mathbf{x}'\|^{\gamma-m-1}=C\|\mathbf{x}-\mathbf{x}'\|^{\gamma-m}\|f\|_{\Lambda_k^{\gamma_1}}.
        \end{split}
    \end{equation}
    Finally, utilizing Lemma~\ref{lem:3}, we have
    \begin{equation}\label{eq:I_333}
        \begin{split}
            I_3 &\leq C\|\mathbf{x}-\mathbf{x}'\|\|f\|_{L^{\infty}}\int_{1}^{\infty}e^{-t}t^{\gamma_2}t^{-m-1}\frac{dt}{t} \\& = C'\|\mathbf{x}-\mathbf{x}'\|\|f\|_{\Lambda_k^{\gamma_1}}\leq C\|\mathbf{x}-\mathbf{x}'\|^{\gamma-m}\|f\|_{\Lambda_k^{\gamma_1}}.
        \end{split}
    \end{equation}
    Collecting~\eqref{eq:I_111},~\eqref{eq:I_222}, and~\eqref{eq:I_333}, we get \eqref{eq:Lip_gamma-m},  which, together with~\eqref{eq:I_444}, proves~\eqref{eq:J_goal_2}.
\end{proof}

{  
\begin{proof}[Completing the proof of Theorem \ref{teo:main_main}]

 Thanks to Theorems \ref{teo:first_inclusion} and \ref{teo:inclusion_2},  it remains to prove the theorem for $\beta$ being any positive integer. To this end, we consider the identity operator and apply Theorems \ref{teo:first_inclusion} and  \ref{teo:inclusion_2} together with  the interpolation  Theorem \ref{teo:interpolation}. 
   
\end{proof}

}

\end{document}